\documentclass[oneside,english,reqno]{amsart}

\usepackage{algorithm}
\usepackage{algorithmic}
\usepackage{multirow}
\usepackage{subfigure}
\usepackage{enumerate}

\usepackage[T1]{fontenc}
\usepackage[latin9]{inputenc}
\usepackage{babel}
\usepackage{amsmath}
\usepackage{amsthm}
\usepackage{amssymb}

\usepackage{comment}

\usepackage{graphicx}

\usepackage{graphicx}
\usepackage{tikz}
\usetikzlibrary{shapes.misc}
\usepackage{float}

\usepackage[unicode=true,pdfusetitle,
bookmarks=true,bookmarksnumbered=false,bookmarksopen=false,
breaklinks=false,pdfborder={0 0 1},backref=false,
hidelinks]
{hyperref}

\makeatletter
\numberwithin{equation}{section}
\numberwithin{figure}{section}

\makeatother

\newtheorem{theorem}{Theorem}[section]

\newtheorem{corollary}[theorem]{Corollary}

\newtheorem{definition}[theorem]{Definition}
\newtheorem{example}[theorem]{Example}

\newtheorem{lemma}{Lemma}[section]
\newtheorem{fact}{Fact}[section]

\newtheorem{proposition}[theorem]{Proposition}
\newtheorem{remark}[theorem]{Remark}

\newcommand{\hilbertH}{{ E}}

\newcommand{\rank}{\text{rank}\,}

\begin{document}

\title[On tangent cone to systems of inequalities and equations]{On tangent cone to systems of inequalities and equations  in  Banach spaces under relaxed constant rank condition
}


\author{
	Ewa M. Bednarczuk$^1$
}
\author{
	Krzysztof W. Le\'sniewski$^2$
}
\author{
	Krzysztof E. Rutkowski$^3$ 
}
\thanks{$^1$ Warsaw University of Technology, 00-662 Warszawa, Koszykowa 75,\\ Systems Research Institute of PAS, PAS, 01-447 Warsaw, Newelska 6, 		
}
\thanks{$^2$ Systems Research Institute of the Polish Academy of Sciences,
}
\thanks{$^3$ Cardinal Stefan Wyszy\'nski University, 01-815 Warsaw, Dewajtis 5.}

\begin{abstract}
Under the relaxed constant rank condition, introduced by Minchenko and Stakhovski, we prove that the linearized cone is contained in the  tangent cone (Abadie condition) for sets represented as solution  sets to systems of finite numbers of inequalities and equations given by continuously differentiable functions defined on Banach spaces.
\end{abstract}
\maketitle

\section*{Introduction}
Conditions ensuring the equality between the tangent  and the linearized cones to the constraint set are at the core of 
optimality conditions in constrained  optimization.  
Let $\hilbertH$ be a Banach space and 
\begin{equation}\label{set:F}
\begin{array}{ll}		
{\mathcal F}:=\{ x\in \hilbertH \mid h_i(x)=0,\ i\in I_0,\ h_i(x)\leq 0,\ i\in I \},
\end{array}
\end{equation}
where $h_i:\ \hilbertH \rightarrow \mathbb{R}$, $i\in I_0\cup I$ are  $C^1$ functions in a neighbourhood of $x_0\in {\mathcal F}$.  Sets $I_0, I$ are finite,   $I_0\cup I=\{1,2,\dots,n\}$, we admit either $I_0$  or $I$ to be empty.

Abadie condition has been introduced in \cite{MR0218116}. It says that the tangent and linearized cone coincide (see section \ref{section:tangent_and_linearized_cones}. In \cite{directional_derivative_Janin} it was shown that CRCQ implies Abadie constraint qualification. In the case $I_0=\emptyset$ and $h_i:\ \mathbb{R}^n\rightarrow \mathbb{R}$, $i\in I$ are convex  it was shown in \cite[Theorem 3.5]{MR1479609} that Abadie condition is equivalent to the metric regularity of the respective set-valued mapping. 

In finite-dimensional settings relationships between constant rank constraint qualification and Abadie condition were investigated in \cite{MR2968256,constant_rank_constraint_Andreani}, and for relationships between relaxed Mangasarian-Fromovitz and Abadie condition see e.g., \cite{Kruger2014}. When $h_i$, $i\in I_0\cup I$ are of class $C^2$, the Abadie condition has been investigated in 
\cite{MR3579742}.
The survey of the existing finite-dimensional constraint qualification ensuring the Abadie conditions can be found in \cite{constraint_qualificatuions_solodov}.	

In infinite dimensional case, the most commonly used constraint qualification is  Robinson's condition (see \cite{MR3919415,Kurcyusz1976,MR0410522}) 
and the relationship  to Lagrange multipliers (see \cite{Bonnans_Shapiro,BFb0004508,MR526427}).

Due to  limitations of  applicability of  Robinson's condition  there exist  a number of recent attempts to define other constraint qualification without referring to active index set (for example various forms of cone-continuity properties in \cite{2019arXiv191206531B} and the references therein).

In \cite[Proposition 1]{MR2179245}, in metric spaces, it was shown that calmness implies Abadie condition.	Recently, the Abadie condition on manifolds was investigated in \cite{2018arXiv180406214B}.

In the present investigation we consider constrained optimization problems defined on  Banach spaces with finite number of constraints. A natural example of such an optimization problem is the projection in Hilbert space onto set of the form \eqref{set:F}.




The aim of the present paper is to investigate the relationship between the Relaxed Constant Rank Constraint Qualification (RCRCQ, Definition \ref{def:rcrcq}), the form of the tangent cone to ${\mathcal F}$ at a given $x_0\in {\mathcal F}$ (the Abadie condition) and the existence of  Lagrange multipliers to the problem
\begin{equation}\tag{P}\label{problem:P}
\begin{array}{ll}
\text{minimize } & h_0(x) \\
\text{s.t. } & x\in {\mathcal F},
\end{array}
\end{equation}
$h_0:\ \hilbertH\rightarrow \mathbb{R}$, in Banach space setting.

In the finite dimensional setting, when $\hilbertH=\mathbb{R}^n$, this question has been discussed in Theorem 1 of \cite{MR2801389} (see also \cite{Minchenko2019}). When dealing with the infinite-dimensional case our main tools are Banach space versions of local representation theorem (Theorem \ref{theorem:local_representation}, \cite[Theorem 2.1.15]{manifolds_tensor_vol2}),  rank theorem (Theorem \ref{theorem:rank}, \cite[Theorem 2.5.15]{manifolds_tensor_vol2}) and Ljusternik theorem (see e.g., \cite[section 0.2.4]{theory_of_external_problems_Ioffe}).

The main results are Theorem \ref{theorem:tangent_cone} and Proposition \ref{proposition:nonmpty_lagrange}.
The proof of   Theorem \ref{theorem:tangent_cone} relies mainly on Proposition \ref{proposition:functional_dependence}. This proposition  can be viewed as a variant of the Implicit Function type theorem and relates constant rank condition with functional independence and functional dependence of a finite number of $C^{1}$ functions defined on a Banach space. 
In  Proposition \ref{proposition:functional_dependence} we use the concept of functional dependence/independence which generalizes to Banach spaces the concept of \cite[Example 2.5.16]{manifolds_tensor_vol2} and the proof of Proposition \ref{proposition:functional_dependence} is based on  	 rank theorem (Theorem \ref{theorem:rank}).

The organization of the paper is as follows. Section \ref{section:constant_rank_condition} is devoted to constant rank condition (CRC).  With the help of CRC, in Section \ref{section:rank_theorem_under_CRC}, we  reformulate  the classical rank theorem in the case considered (finite number of functions defined on a Banach space). In Section \ref{section:functional_dependence} we prove Proposition \ref{proposition:functional_dependence} which is the main tool in deriving our main results of Section \ref{section:tangent_and_linearized_cones}. In Section \ref{section:tangent_and_linearized_cones}, in Definition \ref{def:rcrcq}, we define the relaxed constant rank condition (RCRCQ) and we use it to prove  the Abadie condition (Theorem \ref{theorem:tangent_cone}). In Section \ref{section:functional_dependence_without_CRC} we discuss other concepts of functional dependence and in Section \ref{section:relaxed_constant_rank_and_lagrange_multipliers} in Proposition \ref{proposition:nonmpty_lagrange} we prove the nonemptiness of the Lagrange multipliers set under RCRCQ.

\section{Preliminary facts}\label{section:preliminary}
We start with some preliminary facts and definitions which will be useful in the sequel.

Let $\hilbertH$ be a real Banach space and denote $\hilbertH^*$ its dual. Let $\langle \cdot , \cdot \rangle:\ \hilbertH^* \times \hilbertH\rightarrow \mathbb{R}$ denote the duality mapping between spaces $\hilbertH$, $\hilbertH^*$. We have $\varphi (x)=\langle \varphi  \, , \, x\rangle $ for all $\varphi\in \hilbertH^*$, $x\in \hilbertH$. 
Let $\|\cdot\|$ denote norm on $\hilbertH$ and $\|\cdot\|_*$ denote norm on $\hilbertH^*$.
\begin{definition}
	The closed subspace $H$ of a Banach space $\hilbertH$ splits or is complemented, if there is a closed subspace $G\subset \hilbertH$ such that $\hilbertH=H\oplus G$ (i.e. $\hilbertH=H+ G$ and $H\cap G=\{0\}$), where $\oplus$ denotes the direct sum of $H$ and $G$.
\end{definition}

Let $f:\ U\rightarrow \mathbb{R}^n$, $U\subset \hilbertH$ open set and $x_0\in U$. Fr\'echet derivative $Df(x_0):\ U \rightarrow \mathbb{R}^n$ is a linear operator defined as
\begin{equation*}
\lim_{\Delta x\rightarrow 0} \frac{\|f(x_0+\Delta x)-f(x_0)-Df(x_0) \Delta x\|}{\|\Delta x\|}=0.
\end{equation*}
Moreover, assuming $f=[f_1,\dots,f_n]$, where $f_i:\ U\rightarrow \mathbb{R}$, we write $Df(x_0)=[D f_1(x_0),\dots,D f_n(x_0)]^T$, where $D f_i:\ U \rightarrow \mathbb{R}$.

In the sequel we will use some representations of   the space $\hilbertH$ as a direct sum of closed subspaces.

\begin{fact}\label{fact:existence_dial_basis}(see also Proposition 3.2.1 of \cite{MR3526021})
	Let $\hilbertH$ be a Banach space and consider linearly independent vectors $\{e_1,\dots,e_\kappa\}\in \hilbertH$, $\kappa\in \mathbb{N}$. Let $X_1=\text{span} \, \{e_1,\dots,e_\kappa\}$. 
	Then there exist linearly independent functionals $e_1^*,\dots,e_\kappa^* \in \hilbertH^*$ such that 
	%
	%
	%
	%
	\begin{equation}\label{delta_kroneckera}
	e_i^* e_j = \left\{ \begin{array}{ll}
	1 & \text{if } i=j,\\
	0 & \text{if } i\neq j.
	\end{array}\right.
	\end{equation}
	and $X_1^*:=\text{span}\, \{ e_1^*,\dots, e_\kappa^* \}$ is a $\kappa$-dimensional subspace of $\hilbertH^*$. 
\end{fact}

\begin{proposition}\label{proposition:split} Let $\hilbertH$ be a      Banach space, $f:\hilbertH\rightarrow\mathbb{R}^{\kappa}$, $f=[f_{1},\dots,f_{\kappa}]$, $f_{i}:\hilbertH\rightarrow\mathbb{R}$, $i=1,\dots,\kappa$,  $x_{0}\in\hilbertH$.
	Assume that functionals $D f_i(x_0)$, $i= 1,\dots,\kappa$, are linearly independent. 
	Then $\hilbertH=\hilbertH_1\oplus \hilbertH_2$, where $\hilbertH_2:=\text{ker} \, Df(x_0) ,$
	$ \hilbertH_1:=\text{span}\, \{ Df_1(x_0)^*,\dots, Df_\kappa(x_0)^*\}$ with $Df_i(x_0)^* $, $i=1,\dots,\kappa$, defined as in Fact \ref{fact:existence_dial_basis}.
\end{proposition}

\begin{proof} Taking  $e_{i}:=D f_i(x_0)\in \hilbertH^{*}$, $i= 1,\dots,\kappa$, we get
	$X_1:=\text{span}\, \{ e_1,\dots, e_\kappa\}=
	\text{span}\, \{ D f_1(x_0),\dots, D f_\kappa(x_0)\}$, $\text{dim\,}X_1=\kappa.$
	By 
	Fact \ref{fact:existence_dial_basis}, there exist vectors $e_{i}^{*}:=Df_i(x_0)^*\in  \hilbertH^{**}$, $i=1,\dots,\kappa$ 
	satisfying \eqref{delta_kroneckera}, i.e.
	\begin{equation} 
	\label{delta_1}
	e_{i}^{*}(e_{j})=Df_i(x_0)^*(Df_{j}(x_{0}))=\left\{\begin{array}{ll} 
	1&\text{if  }i=j\\
	0&\text{if  }i=j
	\end{array}
	\right.,
	\end{equation}
	which  are linearly independent and        
	$X_1^*=\text{span}\, \{ Df_i(x_0)^{*},\ 1=1,\dots,\kappa\}\subset \hilbertH^{**}$ is a $\kappa$-dimensional subspace of $\hilbertH^{**}$.
	
	
	Since $X_1$ is the finite-dimensional space  there exists canonical isomorphism (see e.g. By Proposition 10.4 of \cite{linear_algebra_and_opt}) $eval_{X_1}:\ X_1\rightarrow X_1^{**}$ defined as
	\begin{equation*}
	eval_{X_1}(v)(u^*)= u^* (v) \quad \text {for every} \   v\in X_1,\ u^*\in X_1^*
	\end{equation*}
	and we have
	$e_{i}^{*}(e_{j})=e_{j}(e_{i}^{*})=Df_i(x_0)Df_{j}(x_{0})^{*}$ for all $i,j\in \{1,\dots,\kappa\}$. 
	By \eqref{delta_1},
	\begin{equation} 
	\label{delta_2}
	Df(x_{0})(Df_{i}(x_{0})^{*})=v_{i},\ \ \ v_{i}=[0,\dots,\underbrace{1}_{i},\dots,0]\in\mathbb{R}^{\kappa}.
	\end{equation}
	Now we show that $\hilbertH=X_{1}^{*}\oplus \text{ker }Df(x_{0})$. For any  $x\in \hilbertH$,  $Df(x_{0})(x)\in \mathbb{R}^{\kappa}$, and there exist $\alpha_{i}(x)\in \mathbb{R}$, $i=1,\dots,\kappa$  such that
	$Df(x_{0})(x)=\sum_{i=1}^{\kappa}\alpha_{i}(x)v_{i}.$
	Take 
	$m\in X_{1}^{*}$, $m:=\sum_{i=1}^{\kappa}\alpha_{i}(x)Df_{i}(x_{0})^{*}$. By \eqref{delta_2}, we have
	$$
	Df(x_{0})(x-m)=Df(x_{0})(x)-Df(x_{0})(m)=\sum_{i=1}^{\kappa}\alpha_{i}(x)v_{i}-\sum_{i=1}^{\kappa}\alpha_{i}(x)Df(x_{0})(Df_{i}(x_{0})^{*})=0.
	$$
	This shows that $x-m\in\text{ker\,} Df(x_{0})$  which proves the assertion with $\hilbertH_{1}:=X_{1}^{*}$ and $\hilbertH_{2}:=\text{ker\,} Df(x_{0}).$
\end{proof}


\begin{theorem}\label{theorem:local_representation}(\cite[Theorem 2.5.14]{manifolds_tensor_vol2} Local Representation Theorem)
	Let $f:\ U \rightarrow \mathbb{R}^n$ be of class $C^r$, $r\geq1$ in a neighbourhood of $x_0\in U$, $U\subset \hilbertH$ open set. Let $F_1$ be closed split image of $Df(x_0)$ with closed complement $F_2$.
	Suppose that  $Df(x_0)$ has split kernel $\hilbertH_2=\ker Df(x_0)$ with closed complement $\hilbertH_1$. 
	Then there are
	open sets $U_1\subset U \subset \hilbertH_1\oplus \hilbertH_2 $ and $U_2 \subset F_1 \oplus \hilbertH_2$, $x_0\in U_2$  and a $C^r$ diffeomorphism 
	$\psi:\ U_2\rightarrow U_1 $ such that $(f\circ \psi)(u,v)=(u,\eta(u,v))$ for any  $(u,v)\in U_1$, where 
	$u\in \hilbertH_1$, $v\in \hilbertH_2$ and  $\eta:\ U_2\rightarrow \hilbertH_2$ is a $C^r$ map satisfying $D\eta(\psi^{-1}(x_0))=0$.
\end{theorem}

\begin{remark}
	If $\dim \text{range} \,Df(x)=k$ for all $x$ in some neighbourhood $U^\prime(x_0)$, then,  by Inverse Function Theorem (see e.g. \cite[Theorem 2.5.7]{manifolds_tensor_vol2}), there exists an invertible function $\Psi^\prime:\ U^\prime(x_0)\rightarrow U$ such that $f\circ\Psi^\prime$ depends on $k$ variables. 
\end{remark}

\section{Constant rank condition (CRC)}\label{section:constant_rank_condition}

In the present section we recall basic facts related to constant rank condition for $C^{1}$ functions. By rank of $A$ (where $A$ is a finite set of elements of vector space) we will understand the cardinality of  maximally linearly independent subset of elements of $A$.

\begin{definition}\label{def:CRC}
	Let $f_i:\ U\rightarrow \mathbb{R}$, $i=1,\dots,\kappa$ be $C^1$ functions in a neighbourhood of $x_0\in U$, $U\subset \hilbertH$ open set. We say that \textit{constant rank condition} (CRC) holds at $x_0$ if there exists a neighbourhood $V(x_0)$ such that
	\begin{equation*}
		\rank \{ D f_i(x_0),\ i=1,\dots,n \}=const=\rank \{ D f_i(x),\ i=1,\dots,\kappa \}
	\end{equation*} 
	for all $x\in V(x_0)$. We also admit $const=0$, which corresponds to the case  $Df_i(x_0)=0$, $i=1,\dots, \kappa$.
\end{definition}
The constant rank condition appears in the literature \cite[page 127]{manifolds_tensor_vol2}, \cite[page 47]{an_introduction_to_differentialbe_manifolds_Boothby} and \cite[page 503]{mathematical_analysis_Zorich} (under the name \textit{same rank})\footnote{Let us note that the same terminology (constant rank condition) has been already used in \cite{MR2679662} (Definition 1) and is stronger than that proposed in Definition \ref{def:CRC}.}. 

In the sequel we will make a frequent use of the following observation.
\begin{remark} 
	\label{crc-remark}
	If  $D f_{i_j}(x_0)$, $j=1,\dots,k$ are linearly independent, then, by continuity of $D f_{i_j}$, $j=1,\dots,k$, there exists a neighbourhood $U_0(x_0)$ such that elements $D f_{i_j}(x)$, $j=1,\dots,k$ are linearly independent for all $x\in U_0(x_0)$. Additionally,  if we assume  that \text{the constant rank condition (CRC)} holds for $f:=[f_{1},\dots,f_{n}]$ at  $x_0\in \hilbertH$  in some neighbourhood $V(x_0)$ and $\rank \{D f_i(x_0),\ i=1,\dots,n \}=k$, then, for any  $x\in V(x_0)\cap U_{0}(x_0)$ and   $l\in \{1,\dots,n\}\setminus \{i_1,\dots,i_k\}$, the vectors 
	\begin{equation*}
	D f_{i_1}(x),\dots,D f_{i_k}(x),D f_l(x)
	\end{equation*}
	are linearly dependent. 
\end{remark}

Let us note that, when $f=[f_1,\dots,f_n]$, $f_i:U\rightarrow\mathbb{R}$, $i=1,\dots,n$, are of class $C^1$, $U\subset \hilbertH$ open, 
and $\rank \{ D f_i(x_0),\ i=1,\dots,n \}=k$, then $\dim F_{1}=\dim Df(x_0)(\hilbertH) =k$, where 
\begin{equation*}
Df(x_0)y=\left[\begin{array}{c}
\langle D f_1(x_0) \, , \,  y \rangle\\
\vdots\\
\langle D f_n(x_0) \, , \,  y \rangle
\end{array}\right],\quad  y\in \hilbertH.
\end{equation*}
Consequently, $F_1$, $\dim F_1=k$,  splits $\mathbb{R}^n$  and $F_2$ is a closed complement of $F_1$, $\dim F_2=n-k$. 

Moreover, for any $e\in B(0,\delta)\subset \hilbertH$,
\begin{equation*}
\|Df(x_0)e\|:=\left\| \begin{array}{c}
\langle D f_1(x_0) \, , \,  e \rangle\\ \vdots\\  \langle D f_n(x_0) \, , \,  e \rangle
\end{array}
\right\|_1= \sum_{i=1}^n |\langle D f_i(x_0) \, , \,  e \rangle |\leq  \|e\|\sum_{i=1}^n \|D f_i(x_0)\|_*< \varepsilon,
\end{equation*}
whenever $\delta<  \frac{\varepsilon}{ \|D f_i(x_0)\|_*}$, i.e., 
$Df(x_0)$ is continuous, and
$\hilbertH_2:=\ker Df(x_0)=\{ h\in \hilbertH \mid Df(x_0)h=0  \}$ is a closed subspace of $\hilbertH$.
By Proposition \ref{proposition:split}, its complement $\hilbertH_1$ is a closed subspace of $\hilbertH$, moreover, $\dim \hilbertH_1=k$.

\section{Rank theorem under CRC}\label{section:rank_theorem_under_CRC}

In this section we reformulate  the classical rank theorem, Theorem \ref{theorem:rank},
under CRC condition.

Let $x_0\in \hilbertH$ and $\hilbertH_2=\ker Df(x_0)$. By $\hilbertH_1$ we denote its closed complement (see Proposition \ref{proposition:split}).

\begin{lemma}\label{lemma:basis2}
	Let $\hilbertH$ be a 
	Banach space. 
	Let  $f_i:\ U\rightarrow \mathbb{R},\ i= 1,\dots,\kappa$, $U\subset \hilbertH$ open, be $C^1$ functions in a neighbourhood of $x_0\in U$ and let $f=[f_1,\dots,f_\kappa]$. Assume that functionals $D f_i(x)$, $i= 1,\dots, 
	\kappa$, are linearly independent for $x$ from a neighbourhood $U_0(x_0)$.  Then  for any $x\in U_0(x_0)$, the vectors
	\begin{equation*}
	e^i(x):=Df(x) D f_{i}(x_0)^*,\quad i=1,\dots \kappa,
	\end{equation*}
	form a basis in $\mathbb{R}^\kappa$, where $D f_{i}(x_0)^*\in \hilbertH$, $i=1,\dots \kappa,$	(see Fact \ref{fact:existence_dial_basis}) are such that
	\begin{equation*}
	D f_{i}(x_0)^*(D f_{j}(x_0))=\left\{\begin{array}{ll}
	1 & \text{if } i=j,\\
	0 & \text{if } i\neq j.
	\end{array}		
	\right.
	\end{equation*}
\end{lemma}
\begin{proof}
	The existence of $Df_i(x_0)^*$, $i=1,\dots, \kappa$ is ensured by Fact \ref{fact:existence_dial_basis}. 
	Since  $D f_{1}(x_0),\dots, D f_{\kappa}(x_0)$ are linearly independent, by Fact \ref{fact:existence_dial_basis},  $D f_{1}(x_0)^*,\dots, D f_{\kappa}(x_0)^*$ are linearly independent.
	
	Let $\hilbertH_1:=\text{span}\, \{ Df_1(x_0)^*,\dots,Df_\kappa(x_0)^* \}$. 
	First, let us show that for all $x\in U_0(x_0)$ the mapping	$Df(x)|_{\hilbertH_1}:\ \hilbertH_1\rightarrow Df(x)(\hilbertH)$  is an injection.
	
	We start this by showing that for $x=x_0$ this mapping is an injection. 
	Indeed,  suppose that there exists $e_1,e_2\in \hilbertH_1$, $e_1\neq e_2$ such that $Df(x_0)(e_1)=Df(x_0)(e_2)$. Then $e_1-e_2\in \hilbertH_1$ since $\hilbertH_1$ is a linear space and at the same time $e_1-e_2\in \hilbertH_2$. This contradicts the fact that $e_1\neq e_2$.
	
	By  assumption, for any $x\in U_0(x_0)$ there exists a linear isomorphism 
	\begin{equation*}
	L_x:\ \hilbertH_1\rightarrow\text{span} \{ D f_{1}(x)^*,\dots,D f_{\kappa}(x)^*\}
	\end{equation*}
	defined as $L_x( D f_{j}(x_0)^*)=D f_{j}(x)^*$, $j=1,\dots,\kappa$. 
	
	Hence, for any $x\in U_0(x_0)$, $\hilbertH_1=L_x^{-1}(\text{span} \{ D f_{1}(x)^*,\dots,D f_{\kappa}(x)^*\})$ and 
	$Df(x)|_{\hilbertH_1}:\ \hilbertH_1\rightarrow Df(x)(\hilbertH)$ is an injection (by injectivity of composition of two injective mappings).
	
	
	Take any $x\in U_0(x_0)$ and $\alpha_i(x)$, $i=1,\dots,\kappa$ such that 
	%
	\begin{equation*}
	\alpha_1(x)e^1(x)+\dots+	\alpha_\kappa(x) e^\kappa(x)=0.
	\end{equation*}
	By injectivity of $DF(x)|_{\hilbertH_1}$, $k\in \mathbb{N}$ and by linear independence of $Df_i(x_0)^*$, $i=1,\dots,\kappa$, for all $k$ we are getting the lineary independece of $e^i(x), i=1,2,\dots, \kappa$ for $x\in U_0(x_0).$

	Thus,  $e^j(x)$, $j=1,\dots,\kappa$, form a basis in $\mathbb{R}^\kappa$ for any $x\in U_0(x_0)$.
	
\end{proof}

\begin{proposition}\label{proposition:isomorphism_2}
	Let $\hilbertH$ be a 
	Banach space.  
	Let $f_i:\ U\rightarrow \mathbb{R},\ i= 1,\dots,\kappa$, $U\subset \hilbertH$ open, be $C^1$ functions in a neighbourhood of $x_0\in U$. Then the following statements are equivalent:
	\begin{enumerate}[(i)]
		\item\label{isomporphism:s1_2} CRC holds at $x_0$ 
		\item\label{isomporphism:s2_2} the mapping $Df(x)|_{\hilbertH_1}:\ \hilbertH_1\rightarrow Df(x)(\hilbertH)$ is an isomorphism for $x$ in some neighbourhood of $x_0$.
	\end{enumerate} 
\end{proposition}

\begin{proof} \eqref{isomporphism:s1_2} $\implies$ \eqref{isomporphism:s2_2}.
	Let CRC hold at $x_0$  with a neighbourhood $V(x_0)$. Let $\rank \{ D f_i(x_0), i=1,\dots,\kappa \}=k$. By Remark \ref{crc-remark},  there exist indices $i_1,\dots,i_k \subset \{1,\dots,\kappa\}$, such that
	$i_j\neq i_l$ for $j\neq l$,   
	and  a neighbourhood $U_0(x_0)$ such that 
	\begin{equation}\label{continuity_observation_2}
	D f_{i_1}(x),\dots,D f_{i_k}(x) 
	\end{equation}	 
	form a maximally linearly independent subset of  $\{D f_j(x),\ j=1,\dots,\kappa\}$, $x\in U_0(x_0)$. 
	
	Let $f^1(x):=[ f_{i_1}(x),\dots, f_{i_k}(x)]$.  Clearly,  $\ker Df(x_0)=\ker Df^1(x_0)$, where $\ker Df^1(x_0)=\{ h\in \hilbertH \mid Df^1(x_0) h=0   \}$. By Proposition \ref{proposition:split},  
	\begin{equation*}
	\hilbertH_1=\text{span} \{ D f_{i_1}(x_0)^*,\dots,D f_{i_k}(x_0)^*\}
	\end{equation*}	
	and $\dim \hilbertH_1=k$. By CRC,  $\dim(D f(x)(\hilbertH))=k$ for all $x\in U_0(x_0)$.
	
	Since  $D f_{i_1}(x_0),\dots,D f_{i_k}(x_0)$ are linearly independent, the mapping 
	$Df^1(x)|_{\hilbertH_1}:\ \hilbertH_1\rightarrow Df^1(x_0)(\hilbertH)$ is an injection for all $x\in U(x_0)$ (see proof of Lemma \ref{lemma:basis2}).

	Now we discuss the surjectivity of  $Df(x)|_{\hilbertH_1}:\ \hilbertH_1\rightarrow Df(x)(\hilbertH)$ in a neighbourhood of $x_0$. To this aim we note that it is enough to investigate the surjectivity of $Df^1(x)|_{\hilbertH_1}:\ \hilbertH_1\rightarrow Df^1(x)(\hilbertH)$. 
	
	Let us note that $Df(x)e$, $e\in \hilbertH_1$ is fully determined by $Df^1(x)e$. To see this take $e\in \hilbertH_1$. Then $e=\sum_{j=1}^k \lambda_j(D f_{i_j}(x_0))^*$, where $\lambda_j\in \mathbb{R}$, $j=1,\dots,k$. 
	For any $x\in U_0(x_0)$ we have
	\begin{equation*}
	Df^1(x)e=[ D f_{i_l}(x)   e ]_{l=1}^k=\left[\sum_{j=1}^k \lambda_j  D f_{i_l}(x)D f_{i_j}(x_0)^*   \right]_{l=1}^k .
	\end{equation*}
	Again, by Remark \ref{crc-remark}, $D f_l(x)^*$, $l\in \{1,\dots,\kappa\}\setminus \{i_1,\dots,i_k\}$   depend linearly on  $D f_{i_1}(x)^*,\dots,D f_{i_k}(x)^*$, $x\in U_0(x_0)\cap V(x_0)$, we have that
	\begin{equation}\label{representation_dependent_2}
	D f_l(x) e=\sum_{j=1}^k \alpha_{j}^l(x) Df_{i_j}(x)  e ,
	\end{equation}
	where $\alpha_j^l(x)\in \mathbb{R}$, $j\in {1,\dots,k}$, $l\in \{1,\dots,\kappa\}\setminus \{i_1,\dots,i_k\}$, $x\in U_0(x_0)\cap V(x_0)$ and
	\begin{equation*}
	D f_l(x) =\sum_{j=1}^k \alpha_{j}^l(x)  D f_{i_j}(x) .
	\end{equation*} 

	Now we show the surjectivity of $Df(x)|_{\hilbertH_1}:\ \hilbertH_1\rightarrow Df(x)(\hilbertH)$ for $x$ in some neighbourhood of $x_0$. By Lemma \ref{lemma:basis2}, there exists a neighbourhood $U_1(x_0)$ such that the vectors
	\begin{equation*}
	e^j(x):=Df^1(x) D f_{i_j}(x_0)^*,\quad j=1,\dots k,
	\end{equation*}
	form a basis in $\mathbb{R}^k$.	
	Let $x\in U_0(x_0)\cap U_1(x_0)\cap V(x_0)$, $g\in Df(x)(\hilbertH)$ and for $l\in \{1,\dots,\kappa\}$ let us denote by $g_l$ its $l$-th component. By \eqref{representation_dependent_2}, we have $g_l=\sum_{j=1}^{k}\alpha_j^l(x) g_{i_h}$, $l\in \{1,\dots,\kappa\}\setminus \{i_1,\dots,i_k\}$ and, moreover
	\begin{equation*}
	\left[\begin{array}{ccc}
	g_{i_1}\\
	\vdots\\
	g_{i_k}
	\end{array}\right]=\sum_{j=1}^k \beta_j(x) e^j(x),
	\end{equation*}
	for some $\beta_j(x) \in \mathbb{R}$, $j=1,\dots,k$. Hence,
	\begin{align*}
	\left[\begin{array}{ccc}
	g_{i_1}\\
	\vdots\\
	g_{i_k}
	\end{array}\right]&=\sum_{j=1}^k \beta_j(x) Df^1(x) Df_{i_j}(x_0)^*
	=Df^1(x)(\sum_{j=1}^k\beta_j(x)D f_{i_j}(x_0)^*).
	\end{align*}
	And, for $l\in \{1,\dots,\kappa\}\setminus \{i_1,\dots,i_k\}$,
	\begin{equation*}
	g_l=  \sum_{h=1}^{k}\alpha_h^l(x) D f_{i_h}(x) ( \sum_{j=1}^k \beta_j(x)D f_{i_j}(x_0)^* ).
	\end{equation*}
	
	Observe that $\sum_{j=1}^k \beta_j(x)D f_{i_j}(x_0)^*\in \hilbertH_1$, and hence  $Df(x)|_{\hilbertH_1}:\ \hilbertH_1\rightarrow Df(x)(\hilbertH)$ is surjective for $x\in U_0(x_0)\cap U_1(x_0)\cap V(x_0)$. Since $Df(x)|_{\hilbertH_1}:\ \hilbertH_1\rightarrow Df(x)(\hilbertH)$ is surjection and injection between finite-dimensional spaces, it is a (linear) isomorphism.
	
	%
	
	\eqref{isomporphism:s2_2} $\implies$ \eqref{isomporphism:s1_2}
	We have $\hilbertH_2=\ker Df(x_0)=\{ h\in \hilbertH \mid Df(x_0) h=0   \}$ and 
	let $k=\rank \{D f_1(x_0),\dots, D f_\kappa(x_0)  \}$. There exists $i_1,\dots,i_k\subset \{1,\dots,\kappa\}$ such that the system $D f_{i_1}(x_0),\dots, D f_{i_k}(x_0) $
	forms a maximally linearly independent subset of set $ \{D f_1(x_0),\dots, D f_\kappa(x_0)  \}$. Moreover,
	\begin{equation*}
	\hilbertH_1=\hilbertH_2^\perp=\text{span} \{ D f_1(x_0)^*,\dots, D( f_\kappa(x_0))^*  \}=\text{span} \{ D f_{i_1}(x_0)^*,\dots, Df_{i_k}(x_0)^*\}.
	\end{equation*}	
	Since $f_{i_1},\dots, f_{i_k}$ are of class $C^1$ there exists a neighbourhood $U_0(x_0)$ such that $D f_{i_1}(x),\dots,D f_{i_k}(x)$ are linearly independent. 
	
	By assumption $\dim Df(x)(\hilbertH)=k$ for all $x$ in a neighbourhood $V_1(x_0)\subset U_0(x_0)$. Thus $\{D f_{i_1}(x)^*,\dots,D f_{i_k}(x)^*\}$ forms a maximally linearly independent subset of set $\{D f_{1}(x)^*,\dots,Df_{\kappa}(x)^*\}$, $x\in V_1(x_0)$. Hence CRC holds for functions $f_i:\ \hilbertH\rightarrow \mathbb{R},\ i= 1,\dots,\kappa$ at $x_0$ with neighbourhood $V_1(x_0)$.
\end{proof}

In view of Proposition \ref{proposition:isomorphism_2} 
,
Theorem 2.5.15 of \cite{manifolds_tensor_vol2}  (for the finite dimensional case see \cite{analysis_Maurin} and  \cite{mathematical_analysis_Zorich}) takes the following form in the case considered in the present paper.
\begin{theorem}[Rank theorem  under CRC]\label{theorem:rank} Let $\hilbertH$ be a 
	Banach space. 
	Let $x_0\in U$, where $U$ is an open subset of $\hilbertH$ and $f:\ U\rightarrow \mathbb{R}^\kappa$, $f=[f_1,\dots,f_\kappa]$, $f_i:\ \hilbertH\rightarrow \mathbb{R},\ i= 1,\dots,n$ be $C^1$ functions in a neighbourhood of $x_0$. Assume that
	CRC holds at $x_0$ with a neighbourhood $V(x_0)$ and the constant rank  $k$. As previously, let  $\hilbertH_2=\ker Df(x_0)$, and let $\hilbertH_1$ be its closed complement. Then there exist open sets $U_1\subset \mathbb{R}^k\oplus \hilbertH_2$, $U_2\subset \hilbertH$, $V_1\subset \mathbb{R}^\kappa$, $V_2\subset \mathbb{R}^k\oplus \hilbertH_2$ and diffeomorphisms of class $C^1$, $\varphi:\ V_1\rightarrow V_2$ and $\psi:\ U_1\rightarrow U_2$,  $x_0=(x_{01},x_{02})\in U_2\subset U\subset \hilbertH_1\oplus \hilbertH_2$, i.e. $x_{01}\in \hilbertH_1$, $x_{02}\in \hilbertH_2$, $f(x_0)\in V_1$ satisfying 
	\begin{equation*}
	(\varphi \circ f \circ \psi )(w,e)=(w,0),\quad \text{where}\ w\in \hilbertH_1,\ e\in \hilbertH_2
	\end{equation*} for all $(w,e)\in U_1$.
\end{theorem}

\section{Functional dependence}\label{section:functional_dependence}
In this section, by exploiting Theorem \ref{theorem:rank}, we prove Proposition \ref{proposition:functional_dependence} which is an important  tool in the proof of Theorem \ref{theorem:tangent_cone} and can be viewed as a variant of the classical Implicit Function Theorem.

To this aim we extend to Banach spaces the definition of functional dependence of functions $f_i:\ U\rightarrow \mathbb{R}$, $i=1,\dots,\kappa$ , $U\subset \hilbertH$ open, at some $x_0\in U$ given in  \cite[Example 2.5.16]{manifolds_tensor_vol2}.

\begin{definition}\label{def:manifolds} 
	Let $U\subset \hilbertH$ be an open set and let  functions $f_i:\ U\rightarrow\mathbb{R}$, $i=1,\dots,\kappa$ be of class $C^1$  in a neighbourhood of $x_0\in U$. Functions $f_1\dots,f_\kappa$ are  \textit{functionally dependent} at $x_0$ if there exist neighbourhoods $U(x_0)$,  $V(y_0)$, where $y_0:=(f_1(x_0),\dots,f_\kappa(x_0))\in \mathbb{R}^\kappa$ and a function $F:\ V(y_0)\rightarrow \mathbb{R}$ of class $C^1$ such that
	\begin{enumerate}
		\item $F(f_1(x),\dots,f_\kappa(x))=0$	for all $x\in U(x_0)$,
		\item $DF(y_0)\neq 0$.
	\end{enumerate}	
	Functions $f_i:\ U\rightarrow\mathbb{R}$, $i=1,\dots,\kappa$,  are \textit{functionally independent} at $x_0$ if $f_i:\ U\rightarrow\mathbb{R}$, $i=1,\dots,\kappa$ are not functionally dependent at $x_0$, i.e. for any neighbourhoods $V(y_0)$,  $U(x_0)$,  and for any $F:\ V(y_0)\rightarrow \mathbb{R}$ of class $C^1$, if $F(f_1(x),\dots,f_\kappa(x))=0$ for all $x\in U(x_0)$, then $DF(y_0)=0$.
\end{definition}

Now we discuss  conditions ensuring functional dependence/independence.

In the proposition below we generalize Proposition 1 of Section 8.6.3 of \cite{mathematical_analysis_Zorich} to the case, where the argument space is a   
Banach space and the functional dependence is understood in the sense of Definition \ref{def:manifolds}. Assertion (2) of the proposition below establishes a connection of CRC at $x_0$ and  Implicit Function Theorem.
\begin{proposition}\label{proposition:functional_dependence}
	Let $\hilbertH$ be a 
	Banach space. 
	Let $x_0\in U$, $U\subset \hilbertH$, $U$ - open, $f=[f_1,\dots,f_\kappa]$, $f_i:\ U\rightarrow \mathbb{R},\ i= 1,\dots,\kappa$ be $C^1$ functions in a neighbourhood of $x_0$. Assume that 
	CRC holds at $x_0$ with a neighbourhood $V(x_0)$, i.e. 
	\begin{equation*}
	\rank \{ D f_i(x),\ i=1,\dots,\kappa  \}=\rank \{ D f_i(x_0),\ i=1,\dots,\kappa  \}=k \quad \forall x\in V(x_0).
	\end{equation*}
	Let $\hilbertH_2=\ker Df(x_0)$ and $\hilbertH_1$ be its closed complement. Let $i_1,\dots,i_k \subset \{1,\dots,\kappa\}$ be such that
	$i_j\neq i_l$ for $j\neq l$  and $D f_{i_1}(x_0),\dots,D f_{i_k}(x_0)$ are linearly independent.
	\begin{enumerate}[(1)]
		\item\label{proposition:functional_dependence:statement1} If $k=\kappa$, then functions $f_1,\dots,f_\kappa$ are functionally independent at $x_0$.
		\item  If $k<\kappa$, then for any $l\in \{1,\dots,\kappa\}\setminus \{i_1,\dots,i_k\}$ functions  $f_{i_1}\dots,f_{i_k},f_{l}$ are functionally dependent at $x_0$ and there exists a  function $g_l:\ \mathbb{R}^k\rightarrow \mathbb{R}$ of class $C^1$ such that for any $x$ in some neighbourhood of $x_0$
		\begin{equation*}
		f_l(x)=g_l(f_{i_1}(x),\dots, f_{i_k}(x)). 
		\end{equation*}		 
	\end{enumerate}
	
\end{proposition}
\begin{proof}\mbox{\ }	
	\begin{enumerate}[(1)]
		\item The proof follows the line of the proof of Propositon 1 of Section 8.6.3 of \cite{mathematical_analysis_Zorich}.
		
		Let $f=[f_1,\dots,f_\kappa]$. By  Theorem \ref{theorem:rank}, there exist diffeomorphisms of class $C^1$, $\varphi:\ V_1\rightarrow V_2$ and $\psi:\ U_1\rightarrow U_2$ such that
		\begin{equation*}
		(\varphi \circ f \circ \psi )(w,e)=(w,0) \quad\text{ for all $(w,e)\in U_1\subset \hilbertH_1\oplus  \hilbertH_2$}.
		\end{equation*}
		Since $\varphi, \psi$ are diffeomorphisms we have
		\begin{equation*}
		f=\varphi^{-1}\circ (\varphi \circ f \circ \psi )\circ \psi^{-1},
		\end{equation*}
		and hence, $y_0:=f(x_0)$ is an interior  point (in space $\mathbb{R}^\kappa$) of the image of a neighbourhood of $x_0\in \hilbertH$ (note that $\hilbertH_2=\{0\}$). Thus, for any function $F$, the relation 
		\begin{equation*}
		F(f_1(x),\dots,f_\kappa(x))\equiv 0
		\end{equation*}
		holds in a neighbourhood of $x_0$ only if
		\begin{equation*}
		F(y_1,\dots,y_\kappa)\equiv 0
		\end{equation*}		
		in an neighbourhood of $y_0$. Hence, $DF(y_0)= 0$. 
		\item 	 The proof follows the lines of the proof of \cite[Theorem 2.5.12]{manifolds_tensor_vol2}.
		
		If  $\{1,\dots,\kappa\}\setminus \{i_1,\dots,i_k\}= \emptyset$,  the assertion is automatically satisfied. Suppose that $\{1,\dots,\kappa\}\setminus \{i_1,\dots,i_k\}\neq \emptyset$. Without loss of generality we assume that $i_j=j$, $j=1,\dots,k$. 
		
		By Theorem \ref{theorem:rank}, there exist open sets $U_1\subset \mathbb{R}^k\oplus \hilbertH_2$, $U_2\subset \hilbertH$, $V_1\subset \mathbb{R}^\kappa$, $V_2\subset \mathbb{R}^k \oplus \hilbertH_2$ and diffeomorphisms of class $C^1$, $\varphi:\ V_1\rightarrow V_2$ and $\psi:\ U_1\rightarrow U_2$,  $x_0=(x_{01},x_{02})\in U_2\subset U\subset \hilbertH_1\oplus \hilbertH_2$, i.e. $x_{01}\in \hilbertH_1$, $x_{02}\in \hilbertH_2$, $f(x_0)\in V_1$ satisfying 
		\begin{equation}\label{representation:rank_theorem}
		(\varphi \circ f \circ \psi )(w,e)=(w,0),\quad \text{where}\ w\in \mathbb{R}^k,\ e\in \hilbertH_2
		\end{equation} for all $(w,e)\in U_1$. Note that the diffeomorphism $\psi$ is the same as in Theorem \ref{theorem:local_representation}.		
		Hence,
		\begin{equation}\label{eq:does_not_depend}
		\bar{f}(w,e):=(f\circ \psi) (w,e) =\varphi^{-1} (w,0)=(w,\eta(w,e)),
		\end{equation}
		where $\eta:\ \mathbb{R}^k\times \hilbertH_2\rightarrow \hilbertH_2$ is the same as in Theorem \ref{theorem:local_representation}. Thus, $\bar{f}$ does not depend on $e\in \hilbertH_2$.
		
		Let $x\in U_2$ and  denote
		\begin{equation}
		y_i=f_i(x),\quad i=1,\dots,\kappa.
		\end{equation}  There exists $u=(w,e)\in U_1\in \mathbb{R}^k\oplus \hilbertH_2$  such that
		$
		x=\psi(u)
		$.
		Hence, 
		\begin{equation*}
		\begin{array}{ll}
		y_{j}=f_{j}( \psi (w,e))=w_{j},& j=1,\dots,k.
		\end{array}
		\end{equation*}
		For $l\in \{k+1,\dots,\kappa\}$ we have
		\begin{equation*}
		y_l=f_l(x)=f_l(\psi(w,e))=\bar{f}_l(w,e)=\bar{f}_l(y_1,\dots,y_k,e).
		\end{equation*}
		In consequence, by \eqref{eq:does_not_depend}, $y_l=\bar{f}_l(y_1,\dots,y_k)$, $l\in \{k+1,\dots,\kappa\}$. Hence, for any $x\in U_2$, $f_l(x)=\bar{f}_l(f_1(x),\dots,f_k(x))$, $l\in \{k+1,\dots,\kappa\}$.
	\end{enumerate}
\end{proof}

The following example illustrates functional independence of functions at $x_0$ under CRC.

\begin{example}
	Let $\ell_2$ be the Hilbert space of square summable series. Let  $f_1,f_2:\ \ell_2 \rightarrow \mathbb{R}$ be given as	
	$f_1(x)=x_1$, $f_2(x)=x_2$, where $x=(x_1,x_2,\dots)\in \ell_2$. We will show that $f_1,f_2$ are functionally independent at $x_0=0\in \ell_2$. Suppose, by contrary, that $f_1,f_2$ are functionally dependent, i.e there exists a function $F:\ \mathbb{R}^2\rightarrow \mathbb{R}$ of class $C^1$ such that $F(f_1(x),f_2(x))=0$ for all $x$ in some neighbourhood of $0\in \ell_2$ and $DF\neq 0$  in some neighbourhood of $(f_1(0),f_2(0))=(0,0)\in \mathbb{R}^2$. Indeed, by Implicit Function Theorem (see e.g. \cite[Theorem 2.5.7]{manifolds_tensor_vol2}),
	there exists $g:\ \mathbb{R}\rightarrow \mathbb{R}$ such that $f_2(x)=g(f_1(x))$ for all $x$ in some neighbourhood of $x_0$ and
	\begin{equation*}
	D f_2(x)=\left[\begin{array}{c}
	\frac{d g}{d f_1}(f_1(x)) \cdot\frac{\partial f_1}{\partial x_1}\\
	\frac{d g}{d f_1}(f_1(x)) \cdot \frac{\partial f_1}{d x_2}\\
	0\\
	\vdots
	\end{array}\right]^T=\frac{d g}{d f_1}(f_1(x)) D f_1(x),
	\end{equation*}
	i.e. $D f_1(x), D f_2(x)$ are linearly dependent for all $x$ in some neighbourhood of $0\in \ell_2$, which is not true.
\end{example}

\section{Tangent and linearized cones}\label{section:tangent_and_linearized_cones}
In the present section we prove our main result, namely the Abadie condition in  Banach space for the set ${\mathcal F}$ given by \eqref{set:F} under RCRCQ. The finite-dimensional case has been proved by Minchenko and Stakhovski in \cite{MR2801389}.

Let $\hilbertH$ be a 
Banach space. 
Let $C$ be a subset of $\hilbertH$ and $x_0\in \text{cl}\, C$. We use the classical definition of tangent cone of $C$ at $x_0$, 
\begin{align*}
T_C(x_0):=\{ &d\in \hilbertH \mid \exists \varepsilon>0\ \\
&\exists\ \text{a vector function}\ o(t)\ \text{such that}\ \|o(t)\|t^{-1}\rightarrow 0,\ \text{as}\ t\downarrow 0\\
&\text{and}\ x_0+td+o(t)\in C\ \forall\ 0 \leq t\leq  \varepsilon \}.
\end{align*}

For the set ${\mathcal F}$ given by \eqref{set:F} and $x_0\in {\mathcal F}$, the linearized cone is given as
\begin{equation*}
\Gamma_{\mathcal F}(x_0):=\{ d\in \hilbertH \mid \langle D h_i(x_0)\mid d \rangle \leq 0,\ i\in I(x_0),\ \langle D h_i(x_0)\mid d \rangle = 0,\ i\in I_0   \},
\end{equation*}
where $I(x_0):=\{i\in I \mid h_i(x_0)=0 \}$ is the active index set of ${\mathcal F}$ at $x_0$.

\begin{definition}[Relaxed Constant Rank Constraint Qualification]\label{def:rcrcq}
	The \textit{relaxed constant rank constraint qualification} (RCRCQ)  holds for  set ${\mathcal F}$, given by \eqref{set:F} at  $\bar{x}\in {\mathcal F}$, if there exists a neighbourhood $U(\bar{x})$ of $\bar{x}$ such that, for any index set $J$, $I_0\subset J\subset I_0\cup I(\bar{x})$, for every $x\in U(\bar{x})$, the system of vectors $\{ D h_i(x), i\in J  \}$ has constant rank. Precisely, for any $J$, $I_0\subset J\subset I_0\cup I(\bar{x})$,
	\begin{equation*}
	\rank(D h_i(x),i\in J)=\rank(D h_i(\bar{x}),i\in J)\quad \text{for all }x\in U(\bar{x}).
	\end{equation*} 
\end{definition}

\begin{remark} 
	Note that  RCRCQ holds for ${\mathcal F}$ at $x_{0}\in {\mathcal F}$ if and only if  for any index set $J$, $I_0\subset J\subset I_0\cup I(x_0)$, CRC holds at $x_0$ for functions $h_i$, $i\in J$. 
\end{remark}

In Theorem \ref{theorem:tangent_cone} we will use Ljusternik theorem (see \cite[section 0.2.4]{theory_of_external_problems_Ioffe}).
\begin{theorem}(Ljusternik Theorem)\label{theorem:Ljusternik}
	Let $X$ and $Y$ be Banach spaces, let $U$ be a
	neighborhood of a point $x_0\in X$, and let $F:\ U \rightarrow Y$ be a Fr\'echet differentiable mapping. Assume that $F$ is regular at $x_0$, i.e., that
	$\text{Im}\, DF(x_0)= Y$,
	and that its derivative is continuous at this point (in the uniform operator
	topology of the space ${\mathcal L}(X, Y)$). Then the tangent space  $T_M(x_0)$ to the set
	\begin{equation*}
	M =
	\{x \in U \mid F(x) =
	F(x_0)\}
	\end{equation*}
	at the point $x_0$ coincides with the kernel of the operator $DF(x_0)$,
	\begin{equation*}
	T_M(x_0) =
	\text{Ker}\, DF(x_0).
	\end{equation*}
	Moreover, if the assumptions of the theorem are satisfied, then there exist a
	neighborhood $U'\subset U$ of the point $x_0$, a number $K>0$, and a mapping
	$\xi \rightarrow x (\xi)$ of the set $U'$ into $X$ such that
	\begin{align*}
	& F(\xi + x(\xi))= F(x_0),\\
	& \|x(\xi)\|\leq K \|F(\xi) - F(x_0)\|
	\end{align*}
	for all $\xi \in U'$.
\end{theorem}

We start with the following technical lemma (see also \cite{MR2801389} for finite-dimensional case).
\begin{lemma}\label{lemma:r(t)} Let $x_0\in {\mathcal F}$, where ${\mathcal F}$ is given by \eqref{set:F} and $d\in \Gamma_{\mathcal F}(x_0)$.
	For any vector function $r:\ (0,1)\rightarrow \hilbertH$ such that $\|r(t)\|t^{-1}\rightarrow 0$, as $t\downarrow 0$, there exists a number $\varepsilon_0>0$ such that
	\begin{equation}\label{condition:nonactive}
	h_i(x_0+td+r(t))<0 \text{ for all } i\in I\setminus I(x_0,d) \text{ and for all } t\in (0,\varepsilon_0),
	\end{equation} 
	where 
	$I(x_0,d):=\{ i\in I(x_0)\mid \langle D h_i(x_0) \, , \,  d \rangle =0 \}$.
\end{lemma}
\begin{proof}
	Let $d\in \Gamma_{\mathcal F}(x_0)$. 
	If $i\in I\setminus I(x_0)$, then $h_i(x_0)<0$ and, therefore,
	\begin{align*}
	h_i(x_0&+td+r(t))=h_i(x_0)+\langle D h_i(x_0+\theta (td+r(t))) \, , \, td+r(t) \rangle  \\
	&=h_i(x_0)+t \langle D h_i(x_0+\theta (td+r(t)) \, , \, d \rangle+t\langle D h_i(x_0+\theta (td+r(t))) \, , \, \frac{r(t)}{t} \rangle<0, 
	\end{align*} 
	where $0\leq \theta \leq 1$ for all sufficiently small $t>0$.
	
	If $i\in I(x_0)\setminus I(x_0,d)$, then
	\begin{equation*}
	h_i(x_0+td+r(t))=h_i(x_0)+t\langle D h_i(x_0) \, , \, d \rangle +o_i(t)=t\langle D h_i (x_0) \, , \, d \rangle + o_i(t),
	\end{equation*}
	where 
	\begin{align*}
	o_i(t)&:=t \langle D h_i(x_0+\theta (td+r(t)) \, , \, d \rangle+t\langle D h_i(x_0+\theta (td+r(t))) \, , \, \frac{r(t)}{t} \rangle\\
	&-t\langle D h_i(x_0) \, , \, d \rangle.
	\end{align*}
	In this case $h_i(x_0+td+r(t))<0$ for sufficiently small $t>0$ since
	\begin{equation*}
	\langle D h_i(x_0) \, , \, d \rangle <0 \text{ and } o_i(t)t^{-1}\rightarrow 0.
	\end{equation*}
	Consequently, $h_i(x_0+td+r(t))<0$, for all $i\in I\setminus I(x_0,d)$ and for all $t\in (0,\varepsilon_0)$, which proves \eqref{condition:nonactive}.
\end{proof}
Let us note that Lemma \ref{lemma:r(t)} is valid also in the case $I(x_0,d)=\emptyset$. Now we are ready to proof our main result.

\begin{theorem}\label{theorem:tangent_cone} Let $\hilbertH$ be a 
	Banach space and ${\mathcal F}\subset \hilbertH$ be given as in \eqref{set:F}.  
	Assume that RCRCQ holds   for ${\mathcal F}$ at $x_0\in {\mathcal F}$. Then Abadie condition holds, i.e. $\Gamma_{\mathcal F}(x_0)=T_{\mathcal F}(x_0)$.
	
	Moreover, for each $d\in T_{\mathcal F}(x_{0})$ there is a vector function $r:\ (0,1)\rightarrow \hilbertH$, $\|r(t)\|/t\rightarrow 0$ when $t\downarrow 0$, such that for all $t$ sufficiently small \begin{equation} 
	\begin{array}{l}
	h_i(x_{0}+td+r(t))=0,\ i\in J(d),\\
	h_{\ell}(x_{0}+td+r(t))\le 0,\ \ell \in I\setminus J(d),
	\end{array} \quad J(d):=I_{0}\cup I(x_{0},d).
	\end{equation} 
	
	Additionally, whenever $J(d)\neq \emptyset$,   $d\in\text{ker } Dh(x_{0}) $,
	where $h:=[h_{i_j}]$, $j=1,\dots,k$,  $$	\rank \{D h_i(x_0+td+r(t)), i\in J(d) \}=\rank \{D h_i(x_0), i\in J(d) \}=k$$ for all $t$ sufficiently small.
	
	%
	
\end{theorem}
\begin{proof}
	The inclusion $T_{\mathcal F}(x_0)\subset \Gamma_{\mathcal F}(x_0)$ is immediate. To see the converse, take any $d\in \Gamma_{\mathcal F}(x_0)$. We start by  considering the case $J:=J(d)\neq  \emptyset$.
	By RCRCQ of ${\mathcal F}$ at $x_0$,  we have
	\begin{equation*}
	\rank \{D h_i(x_0+td+r), i\in J \}=\rank \{D h_i(x_0), i\in J \}=k,
	\end{equation*}
	for $(t,r)$ in some neighbourhood of $(0,0)\in \mathbb{R}\times \hilbertH$. By Remark \ref{crc-remark},   there  exist indices $i_1,\dots,i_k$, such that $D h_{i_1}(x_0+td+r),\dots,D h_{i_k}(x_0+td+r)$ are linearly independent for $(t,r)$ in some neighbourhood of $(0,0)$. Without loss of generality, we can assume that $i_j=j$, $j=1,\dots,k$. 
	
	If $k=n$, i.e. $J\setminus\{1,\dots,k\}=\emptyset$, then, by 
	applying Ljusternik Theorem  \ref{theorem:Ljusternik} to
	\begin{equation*}
	M := \{ x\in \hilbertH \mid h_i(x)=h_i(x_0)=0,\ i\in J=\{1,2,\dots,n\}  \},
	\end{equation*} the conclusion holds. 
	
	If $k<n$ then, by (2) of Proposition \ref{proposition:functional_dependence}, applied to $h_i$, $i\in 1\dots,k$, there exist  functions $g_l$, $l\in J\setminus \{1,\dots,k\}$ of class $C^1$, such that
	\begin{equation}\label{eq:implicit}
	h_l(x_0+td+r)=g_l(h_1(x_0+td+r),\dots,h_k(x_0+td+r)),
	\end{equation} 
	for $(t,r)$ in some neighbourhood of  $(0,0)$.
	
	Consider the system
	\begin{equation}\label{eq:system_1}
	h_i(x_0+td+r)=0,\quad i\in J
	\end{equation} 
	with respect to variables $t,r$.
	Let us note that  system \eqref{eq:system_1} is satisfied for $(t,r)=(0,0).$ 
	
	Obviously, in some neighbourhood of $(0,0)$,  system \eqref{eq:system_1} is equivalent to
	\begin{equation}\label{eq:system_2}\left\{
	\begin{array}{c}
	h_1(x_0+td+r)=0\\
	\dots\\
	h_k(x_0+td+r)=0
	\end{array}\right.
	\end{equation}
	with additional condition
	\begin{equation}\label{eq:dependence}
	h_l(x_0+td+r)=g_l(h_1(x_0+td+r),\dots,h_k(x_0+td+r))=0,\ l\in J\setminus \{1,\dots,k\}.
	\end{equation}
	Note that $g_l(h_1(x_0),\dots,h_k(x_0))=0$, $l\in J\setminus \{1,\dots,k\}$ and therefore $g_l(0,\dots,0)=0$, $l\in J\setminus \{1,\dots,k\}$.
	
	We have 
	\begin{align*}
	\langle D h_i(x_0) \, , \, d \rangle = 0,\quad  i\in J=I_0\cup I(x_0,d). 
	\end{align*}
	
	Hence,  $d\in \ker Dh(x_0)$, where $h(x)=[h_1(x)\dots,h_k(x)]$. By applying Ljusternik Theorem \ref{theorem:Ljusternik} with $F=h$ at $x_0$, we obtain that $d\in T_{M}(x_0)$, where
	\begin{equation*}
	M:=\{ x \in \hilbertH \mid h(x)=0 \}.
	\end{equation*}
	This means that there exist $\varepsilon>0$ and a function $r:\ [0,\varepsilon)\rightarrow \hilbertH$, $\|r(t)\|t^{-1}\rightarrow 0$, $t\downarrow 0,$ such that 
	\begin{equation*}\left\{
	\begin{array}{c}
	h_1(x_0+td+r(t))=0\\
	\dots\\
	h_k(x_0+td+r(t))=0.
	\end{array}\right.
	\end{equation*}
	By \eqref{eq:dependence}, $h_i(x_0+td+r(t))=0$, $i\in J$ for $t\in [0,\varepsilon]$. By Lemma \ref{lemma:r(t)}, there exists $\varepsilon_0>0$ such that
	\begin{equation}
	x_0+td+r(t)\in {\mathcal F}\quad t\in [0,\min\{\varepsilon_0,\varepsilon\}].
	\end{equation}
	Thus, $d\in T_{\mathcal F}(x_0)$.
	
	Now, let us consider the  case $J=\emptyset$ (i.e. the case when both $I_0=\emptyset$ and $I(x_0,d)=\emptyset$). Then, by Lemma 5.1, for any vector function $r:\ (0,1)\rightarrow \hilbertH$, $\|r(t)\|/t\rightarrow 0$ when $t\downarrow 0$ there exists there exists $\varepsilon>0$ such that
	\begin{equation}
	x_0+td+r(t)\in {\mathcal F}\quad t\in [0,\varepsilon],
	\end{equation}
	i.e.,  $d\in T_{\mathcal F}(x_0)$. 
\end{proof} 

The following corollary refers to the special case, where there is no inequality constraints in the definition of the set ${\mathcal F}$.
\begin{corollary}
	Suppose that $I=\emptyset$, i.e. there is no inequalities in the representation \eqref{set:F} of the set ${\mathcal F}$ i.e.
	$$
	{\mathcal F}=\{x\in \hilbertH \mid h_{i}(x)=0\ \ i=1,\dots,n\}
	$$ and CRC holds at $x_{0}\in {\mathcal F}$, i.e. 
	there  exists a neigbourhood $U(x_0)$ s.t. 
	$$
	\rank \{D h_i(x_0), i=1,2,\dots, n \}=\rank \{D h_i(x), i= 1,2,\dots,n \}=k
	$$
	for all $x\in U(x_0)$.
	Then 
	$$
	T_{\mathcal F}(x_0)= \{ d: \langle D h_i(x_0) \, , \, d\rangle =0,\ i=1,2,\dots,n\}.
	$$	
	Moreover, if $I_k=\{i_1, i_2, \dots, i_k\}$  is such that $D h_{i_j}(x_0)$, $i_j\in I_k$ are linearly independent, then   
	\begin{equation*}
	T_{\mathcal F}(x_0)=\text{ker} \left[ \begin{matrix} D h_{i_1}(x_0)\\\vdots  \\D h_{i_k}(x_0) \end{matrix}  \right].
	\end{equation*}
\end{corollary}
\begin{proof}
	
	By assumption, for any	 $\ell\notin I_k$ 
	$$
	D h_{\ell}(x_0)= \sum\limits_{i\in I_k} \lambda_i^{\ell} D h_{i}(x_0).
	$$
	
	This shows that $T_{\mathcal F}(x_0)$ does not depend upon the choice of the set  $I_k$. 
\end{proof}

\section{Functional dependence/independence without CRC}\label{section:functional_dependence_without_CRC}

In our main theorem (Theorem  \ref{theorem:tangent_cone}) we used constant rank condition (and RCRCQ) to be able to apply Proposition \ref{proposition:functional_dependence}, where we used the concept of functional dependence/independence according to Definition \ref{def:manifolds}.  

In this additional section we investigate functional dependence/independence with respect to Definition \ref{def:manifolds} without CRC. Moreover, in Subsection \ref{subsection_functional_dependence_others} we review  other most common concepts of functional dependence/independence (Definitions \ref{def:Boothby}, \ref{def:Raghavan}, \ref{def:Laszlo}).

\begin{proposition} Let $\hilbertH$ be a 
	Banach space. 
	Let $x_0\in U$, $U\subset \hilbertH$ open 
	and  $f_1,\dots,f_\kappa:\ U\rightarrow \mathbb{R}$.
	Suppose that in every neighbourhood $U(x_0)$ there exists $x\in U(x_0)$ such that $D f_1(x),\dots,D f_n(x)$ are linearly independent. Then functions $f_1,\dots,f_\kappa$ are functionally independent at $x_0$.
\end{proposition}
\begin{proof}
	Let $F:\ V\rightarrow \mathbb{R}$ be a function of class $C^1$ defined on a neighbourhood $V(y_0)$ such that $F(f_1(x),\dots,f_\kappa(x))=0$ for any $x$ in some neighbourhood $U(x_0)$.
	
	We show that 
	it must be $DF(y_0)=0$, where $y_0=(f_1(x_0),\dots,f_\kappa(x_0))$. By assumption, let $U(x_0)$ be a neighbourhood of $x_0$ and $x^\prime \in U(x_0)$ be such that  $D f_1(x^\prime),\dots,D f_\kappa(x^\prime)$ are linearly independent. 
	
	There exists a neighbourhood  
	$U(x^\prime)\subset U(x_0)$ and $D f_1(z^\prime),\dots,D f_\kappa(z^\prime)$ are linearly independent for all $z\in U(x^\prime)$. By (1) of  Proposition \ref{proposition:functional_dependence}, it must be $DF(f(x^\prime))=0$. By smoothness of function $F$ and $f$, the latter equality implies  $DF(f(x_0))=0$.
\end{proof}
\begin{proposition}(Local stability of functional dependence)\label{local_stability_functional_dependence} Let $\hilbertH$ be a  Banach space. 
	If $f_1,\dots,f_\kappa:U\rightarrow \mathbb{R}$, $U\subset \hilbertH$ open, are functionally dependent at $x_0\in U$, then there exists a neighbourhood $U(x_0)$  such that $f_1,\dots,f_\kappa$ are functionally dependent at any $x\in U(x_0)$.
\end{proposition}
\begin{proof}
	Let $f=[f_1,\dots,f_\kappa]$. Assume that  $f_1,\dots,f_\kappa$ are functionally dependent at $x_0$, i.e. there exist neighbourhoods $U(x_0)$, $V(y_0)$, $y_0=f(x_0)$, and a function $F:\ V(y_0)\rightarrow \mathbb{R}$ such that $DF(y_0)\neq 0$ and $F(f_1(x),\dots,f_\kappa(x))=0$ for all $x\in U(x_0)$.
	
	Since $DF(y_0)\neq 0$ and $F$ is of class $C^1$ there exists a neighbourhood $V_1(y_0)$ such that $DF(y)\neq 0$ for all $y\in V_1(y_0)$. Let $x^\prime \in U(x_0)\cap f^{-1}(V_1(y_0))$,  $U(x^\prime) :=U(x_0)\cap f^{-1}(V_1(y_0))$ and $y^\prime:=f(x^\prime)$. Then function $F:\ V_1(y_0)\cap V(y_0) \rightarrow \mathbb{R}$ satisfies $DF(y^\prime)\neq 0$ and $F(f_1(x),\dots,f_\kappa(x))=0$ for all $x\in U(x^\prime)$. Hence $f_1,\dots,f_\kappa$ are functionally dependent at any $x^\prime\in U(x_0)\cap f^{-1}(V_1(y_0))$.
\end{proof}

The fact below relates functional dependence with linear dependence of gradients.

\begin{fact}\label{fact:dependence_stablility} Let $\hilbertH$ be a 
	Banach space. 
	Let $x_0\in U$, $U\subset \hilbertH$ open, and  $f_1,\dots,f_\kappa:\ U\rightarrow \mathbb{R}$. 
	Suppose that $f_1,\dots,f_\kappa$ are functionally dependent at $x_0$. Then there exists a neighbourhood $U(x_0)$ such that $D f_1(x),\dots, D f_\kappa(x)$, $x\in U(x_0)$, are linearly dependent.
\end{fact}
\begin{proof}
	The proof follows immediately from Remark \ref{crc-remark}, (1) of Proposition \ref{proposition:functional_dependence} and Proposition \ref{local_stability_functional_dependence}.
\end{proof}

The following proposition  provides sufficient conditions for functional independence.
\begin{proposition}\label{proposition:functional_independence2}
	Let $x_0\in U$, $U\subset\hilbertH$ open,  $f_1,\dots,f_\kappa:\ U\rightarrow \mathbb{R}$. 
	If, for any neighbourhood $U(x_0)$,  $\text{int}\, f(U(x_0))\neq \emptyset$,
	then $f_1,\dots,f_\kappa$ are functionally independent at $x_0$.
	
\end{proposition}
\begin{proof} Let $\hilbertH$ be a  Banach space. Let $U(x_0)$ be a neighbourhood of $x_0$ and $V(y_0)$ be a neighbourhood of $y_0=f(x_0)$ and $F:\ V\rightarrow \mathbb{R}$ be of class $C^1$ such that $F(f_1(x),\dots,f_\kappa(x))=0$ for all $x\in U(x_0)$. 
	
	By the  continuity of $f$, for any $m\in \mathbb{N}$  there exists $U_m^\prime(x_0)$ such that $f(U_m^\prime(x_0))\subset B(y_0,\frac{1}{m})$. Let $U_m^{\prime\prime}(x_0)=U(x_0)\cap U_m^\prime(x_0)$. Then, by assumption, $\text{int} (f(U_m^{\prime\prime}(x_0)))\neq \emptyset$ and, moreover, $A_m:= [\text{int} f(U_m^{\prime\prime}(x_0))]\cap B(y_0,\frac{1}{m})=\text{int} f(U_m^{\prime\prime}(x_0))$ is a nonempty open set. Since $F(y)=0$ for all $y\in A_m$ we have $DF(y)=0$ for all $y\in A_m$.
	
	Since $F:\ V\rightarrow \mathbb{R}$ is of class $C^1$, there exists a sequence $y_m\rightarrow y_0$, $y_m\in A_n$, such that $DF(y_m)=0$.
	By the smoothness of $F$, it must be $DF(y_0)=0$. In consequence,  functions  $f_1,\dots,f_\kappa$ are functionally independent at $x_0$.
\end{proof}

\subsection{Functional dependence/independence, other definitions}\label{subsection_functional_dependence_others}
Here we compare the concept of functional dependence given in Definition \ref{def:manifolds} with other concepts of functional dependence appearing in the literature.

Let us note that in the proof of Theorem \ref{theorem:tangent_cone} and Proposition \ref{proposition:functional_independence2} we use the concept of functional dependence as defined in Definition \ref{def:manifolds}. In general, without condition 2. of Definition \ref{def:manifolds} we are not able to deduce formula \eqref{eq:implicit}.

The definition of functional dependence at $x_0$ given  in Chapter II.7 of \cite{an_introduction_to_differentialbe_manifolds_Boothby}  can be rewritten in Banach spaces as follows.
\begin{definition}\label{def:Boothby} 
	Let $\hilbertH$ be a  Banach space. 
	Let $U\subset \hilbertH$ be an open set and let  functions $f_i:\ U\rightarrow\mathbb{R}$, $i=1,\dots,\kappa$ be of class $C^1$ in a neighbourhood of $x\in U$. Functions $f_1\dots,f_\kappa$ are said to be \textit{functionally dependent} at $x_0$ if there exists a neighbourhood $U(x_0)$ and a neighbourhood $V(y_0)$, where $y_0:=(f_1(x_0),\dots,f_\kappa(x_0))\in \mathbb{R}^\kappa$ and a  function $F:\ V\rightarrow \mathbb{R}$ of class $C^1$ such that
	\begin{enumerate}
		\item $F(f_1(x),\dots,f_\kappa(x))=0$	for all $x\in U(x_0)$,
		\item $F\not\equiv 0$ on $V(y_0)$.
	\end{enumerate}	
\end{definition}

Definition \ref{def:Boothby} for continuous functions $f_1\dots,f_\kappa$, $F$ in finite-dimensional settings was given in  Paragraph 8.6.3 of \cite{mathematical_analysis_Zorich}.
\begin{remark} 
	
	
	Clearly, if $f_1,\dots,f_\kappa$ are functionally dependent at $x_0$ in the sense of 
	Definition \ref{def:manifolds},
	then $f_1,\dots,f_\kappa$ are functionally dependent at $x_0$ in the sense of 
	Definition
	\ref{def:Boothby}.
\end{remark}
The example below illustrates the difference between of definitions of functional dependence given in 
Definition \ref{def:manifolds} and
Definition \ref{def:Boothby}.
Let us note that the functions $f_1,f_2$ from the example below  do not satisfy the CRC condition at $x_0=(0,0)$. 
\begin{example}\label{example:functional_dependence}
	Let $f_1,f_2:\ \mathbb{R}^2\rightarrow\mathbb{R}$ be defined as $f_1(x_1,x_2)=x_1^2$, $f_2(x_1,x_2)=x_2^2$ and $x_0=(0,0)$. We will show that $f_1, f_2$ are functionally dependent at $x_0$ in the sense 
	of 
	Definition \ref{def:Boothby}
	and are  functionally independent  at $x_0$ in the sense 
	of Definition \ref{def:manifolds}.
	
	Let $F:\ \mathbb{R}^2\rightarrow \mathbb{R}$ 
	be defined as follows
	\begin{equation*}
	F(y_1,y_2):=\left\{\begin{array}{lll}
	0 & \text{if} & y_1\geq 0 \ \wedge\ y_2\geq 0,\\
	y_1^2 & \text{if} & y_1<0\ \wedge  \ y_2 \geq 0,\\
	y_2^2 & \text{if} & y_1\geq 0 \ \wedge\ y_2<0,\\
	y_1^2+y_2^2 & \text{if} & y_1<0 \ \wedge y_2<0 
	\end{array}  \right. .
	\end{equation*}
	Then $F$ is of class $C^1$ and
	\begin{equation*}
	F(f_1(x_1,x_2),f_2(x_1,x_2))=F(x_1^2,x_2^2)=0
	\end{equation*} 
	for all $(x_1,x_2)\in \mathbb{R}^2$. Moreover, in any neighbourhood of $y_0:=(f_1(x_0),f_2(x_0))=(0,0)$ there exists $y=(y_1,y_2)$ such that $y_1<0$ or $y_2<0$, i.e. $F(y)\neq 0$. Hence, $f_1, f_2$ are functionally dependent at $x_0$ in the sense of 
	Definition \ref{def:Boothby}
	
	Now we show that  $f_1, f_2$ are functionally independent  at $x_0$ in the sense 
	of Definition \ref{def:manifolds}.
	By contrary, suppose, that $f_1, f_2$ are functionally dependent at $x_0$ in the sense
	of Definition \ref{def:manifolds}.
	Then there exists  a function $F:\ \mathbb{R}^2\rightarrow \mathbb{R}$ of class $C^1$ such that $DF(y_0)\neq 0$ at $y_0:=(f_1(x_0),f_2(x_0))=(0,0)$ and $F(x_1^2,x_2^2)=0$ for all $(x_1,x_2)$ in some neighbourhood of $x_0$.
	
	Let $U(x_0)$ be any neighbourhood of $x_0$ and $V(y_0)$ be any neighbourhood of $y_0$. Let $U^\prime(x_0)=\{ (x_1,x_2)\in \mathbb{R}^2 \mid \ (\text{sgn}(x_1)x_1^2, \text{sgn}(x_2)x_2^2)\in U(x_0) \}$, where $\text{sgn}$ is the signum function. Let us take any $y^\prime\in V(y_0)\cap U^\prime(x_0)$ such that $y^\prime=(y_1^\prime,y_2^\prime)$, where $y_1^\prime>0$ and $y_2^\prime>0$. Let $V(y^\prime)=V(y_0)\cap U^\prime(x_0)\cap  \mathbb{R}_{++}^2$, where $R_{++}^2:=\{ (x_1,x_2) \mid x_1>0 \ \wedge x_2>0 \}$. Let us note that $y_0\in \text{cl}(V(y^\prime))$.
	Then for all $y=(y_1,y_2) \in V(y^\prime)$ we have $(\sqrt{y_1},\sqrt{y_2})\in U^\prime(x_0)$ and, moreover, $F(y)=F(\sqrt{y_1}^2,\sqrt{y_2}^2)=0$. Thus, $DF(y)=0$ for all $y \in V(y^\prime)\subset V(y_0)$, which is a contradiction  with the assumption $DF(y_0)\neq 0$.  
\end{example}

\begin{definition}
	Let $\Omega\subset \hilbertH$ be a nonempty set. Subset $A\subset \Omega$ is \textit{nowhere dense} in $\Omega$ if for all $U\subset \Omega$, $U$ open in $\Omega$, $U\neq \emptyset$ there exists $V\subset U$, $V$ open in $\Omega$, $V\neq \emptyset$ such that $A\cap V=\emptyset$, i.e. $A\subset \Omega\setminus V$.
\end{definition}

The definition of functional dependence on a set $\Omega$  given  in Chapter 1 of \cite{lectures_on_topics_in_analysis_Raghavan} is formulated in Banach spaces for $C^{\infty}$ functions. Here we reformulate it in Banach space settings for $C^1$ functions in the following way. 
\begin{definition}\label{def:Raghavan}
	Let $\hilbertH$ be a  Banach space. Let $\Omega\subset \hilbertH$ be an open set and let functions $f_i:\ \Omega\rightarrow\mathbb{R}$, $i=1,\dots,\kappa$ be of class $C^1$ in a neighbourhood of $x_0\in \Omega$. Functions $f_1\dots,f_\kappa$ are said to be \textit{functionally dependent} at $x_0$ if there 
	exist a neighbourhood $U(x_0)\subset \Omega$ and a neighbourhood $V(y_0)$, 
	$y_0:=(f_1(x_0),\dots,f_\kappa(x_0))\in \mathbb{R}^\kappa$, and a function $F:\ V(y_0)\rightarrow \mathbb{R}$ of class $C^1$ such that
	\begin{enumerate}
		\item $F(f_1(x),\dots,f_\kappa(x))=0$ for all $x\in U(x_0)$,
		\item $F^{-1}(0)$ is nowhere dense in $V(y_0)$.
	\end{enumerate}
\end{definition}
The following example  shows that the functional dependence in the sense of Definition \ref{def:Raghavan} does not imply functional dependence in the sense of Definition \ref{def:manifolds}.
\begin{example}\label{example:not_manifolds}
	Let $f_1(t)=t^3$, $f_2(t)=t^2$, $t\in \mathbb{R}$. Then functions $f_1,f_2$ are functionally dependent at $t=0$ in the sense of Definition \ref{def:Raghavan}, since by taking  $F(x,y)=x^3-y^2$ we get:
	\begin{enumerate}
		\item $F(f_1(t),f_2(t))=0$ for all $t\in \mathbb{R}$
		\item $F^{-1}(0)$ is nowhere dense in any neighbourhood of $(0,0)$.
	\end{enumerate}
	On the other hand for any neighbourhood of $U(0)$ and $V((0,0))$ and for any function $F:\ V((0,0))\rightarrow \mathbb{R}$ of class $C^1$ if $F(f_1(t),f_2(t))=0$ for all $t\in U(0)$, then
	\begin{align*}
	DF(0,0)=
	\frac{\partial F (f_1(0),f_2(0))}{\partial f_1} \frac{ d f_1(0)}{dt} + 
	\frac{\partial F (f_1(0),f_2(0))}{\partial f_2} \frac{ d f_2(0)}{dt}= 0.
	\end{align*} 
	Therefore functions $f_1,f_2$ are functionally independent at $t=0$ in the sense of Definition \ref{def:manifolds}.
\end{example}

\begin{proposition}\label{propostion:Raghavan_imples_Boothby}
	Let $\hilbertH$ be a  Banach space. If functions $f_i:\ \Omega\rightarrow\mathbb{R}$, $i=1,\dots,\kappa$ of class $C^1$ are functionally dependent at  $x_0\in \Omega$ in the sense of Definition \ref{def:Raghavan} then they are functionally dependent in the sense of 
	Definition \ref{def:Boothby}.
\end{proposition}
\begin{proof}
	Assume that that $f_1,\dots,f_\kappa$ are functionally dependent at  $x_0\in \Omega$ in the sense of Definition \ref{def:Raghavan}. Then there are neighbourhood $U(x_0)$ and a neighbourhood $V(y_0)$, $V(y_0)\subset f(U(x_0))$, where $y_0:=(f_1(x_0),\dots,f_\kappa(x_0))\in \mathbb{R}^\kappa$ and a function $F:\ V\rightarrow \mathbb{R}$ of class $C^1$ such that
	\begin{enumerate}
		\item $F(f_1(x),\dots,f_\kappa(x))=0$ for all $x\in U(x_0)$,
		\item $F^{-1}(0)$ is nowhere dense in $V(y_0)$.
	\end{enumerate}
	According to the definition of nowhere dense set, for every nonempty open set $U\subset V(y_0)$,  there exists an open set a nonempty set $V\subset U$, such that $V\cap F^{-1}(0)=\emptyset$, i.e.  $F(x)\neq 0$ for every $x\in V$. In conclusion, $F$ satisfies condition 2. of 
	Definition \ref{def:Boothby}.
\end{proof}
\begin{proposition}\label{proposition:Raghavan_to_Laszlo} Let $\hilbertH$ be a 
	Banach space. 
	If functions $f_i:\ \Omega\rightarrow\mathbb{R}$, $i=1,\dots,\kappa$ of class $C^1$ are functionally dependent at  $x_0\in \Omega$ in the sense of Definition \ref{def:Raghavan}, then there exists a neighbourhood $U(x_0)$ such that $D f_1(x), \dots,D f_\kappa(x)$, $x\in U(x_0)$ are linearly dependent.
\end{proposition}
\begin{proof}
	Suppose, by contrary  that for any neighbourhood $U(x_0)$ there exists $x^\prime\in U(x_0)$ such that $D f_1(x^\prime), \dots,D f_\kappa(x^\prime)$ are linearly independent. Then there exists a neighbourhood  $U(x^\prime)$ such that $D f_1(x), \dots,D f_\kappa(x)$, $x\in U(x^\prime)$ are linearly independent. By  Theorem \ref{theorem:rank}, $f(U(x^\prime))$  has a nonempty interior (see e.g. proof of \eqref{proposition:functional_dependence:statement1} of Proposition \ref{proposition:functional_dependence}) and $f(U(x^\prime))\subset f(U(x_0))\subset F^{-1}(0)$, hence $F^{-1}(0)$ would not be nowhere dense.
\end{proof}

The definition of functional dependence at $x_0$ given  in Chapter 4 of \cite{Ordinary_and_partial_diff-eq_for_beginner_Laszlo}   can be rewritten in Banach spaces as follows.
\begin{definition}\label{def:Laszlo}
	Let $\hilbertH$ be a Banach space and let functions $f_i:\ U\rightarrow\mathbb{R}$, $U\subset \hilbertH$ open, $i=1,\dots,\kappa$ be of class $C^1$ in a neighbourhood of $x_0\in U$. Functions $f_1\dots,f_\kappa$ are \textit{functionally dependent} at $x_0$, if 
	\begin{equation}\label{cond:Szekelyhidi}
	\rank \{ D f_i(x_0),\ i=1,\dots,\kappa  \}<\kappa.
	\end{equation}
	Otherwise, we say that $f_1\dots,f_\kappa$ are \textit{functionally independent} at $x_0$. 
	
	Let $\Omega\subset U$ be an open set.  We say that  functions $f_i:\ U\rightarrow\mathbb{R}$, of class $C^1$, $i=1,\dots,\kappa$, are  \textit{functionally dependent on $\Omega$} if \eqref{cond:Szekelyhidi} holds for all $x_0\in \Omega$. Functions $f_i:\ U\rightarrow\mathbb{R}$, of class $C^1$, $i=1,\dots,\kappa$, are  \textit{functionally independent on $\Omega$} 
	if 
	\begin{equation}\label{cond:Szekelyhidi2}
	\rank \{ D f_i(x),\ i=1,\dots,\kappa  \}=\kappa \quad \text{for all}\ x\in \Omega.
	\end{equation}
\end{definition}

In \cite[Theorem 4.1.3]{Ordinary_and_partial_diff-eq_for_beginner_Laszlo} it was shown that for $f_i:\ \mathbb{R}^\kappa\rightarrow \mathbb{R}$, $i=1,\dots,\kappa$, Definition \ref{def:Boothby}  and Definition  \ref{def:Laszlo} are equivalent.

\begin{remark} Let $\hilbertH$ be a 
	Banach space. By Fact \ref{fact:dependence_stablility}, if  functions $f_i:\ U\rightarrow\mathbb{R}$, of class $C^1$, $i=1,\dots,\kappa$, are functionally dependent at $x_0\in U$, $U\subset \hilbertH$ open, in the sense of Definition \ref{def:manifolds}, then there exists a neighbourhood  $U(x_0)$ such that they are functionally dependent on $U(x_0)$ in the sense of Definition \ref{def:Laszlo}.
\end{remark}

The following example illustrates the fact that the functional dependence in the sense of Definition \ref{def:Laszlo} in a neighbourhood of $x_0$ does not imply the functional dependence at $x_0$ in the sense of Definition \ref{def:manifolds}.


\begin{example}\label{example:whirpool}
	Let us consider the \textit{geometric tornado} function $f=[f_1,f_2,f_3]$, $f_1,f_2,f_3:\ \mathbb{R}\rightarrow \mathbb{R}$  defined as follows (see Fig. \ref{fig:whirpool})\\
	\begin{minipage}{0.5\linewidth}
	    	\begin{align*}
		&f_1(x)=\left\{\begin{array}{ll}
		x^3\sin (\frac{1}{x}) & \text{if}\ x\neq 0,\\
		0 & \text{otherwise,}
		\end{array}\right.\\
		&f_2(x)=\left\{\begin{array}{ll}
		x^3\cos (\frac{1}{x}) & \text{if}\ x\neq 0,\\
		0 & \text{otherwise,}
		\end{array}\right.\\
		&f_3(x)=x^3.		
		\end{align*}
	\end{minipage}\hspace{-2cm}
	\begin{minipage}{0.75\linewidth}
		\begin{figure}[H]
			\centering	
			\includegraphics[scale=1]{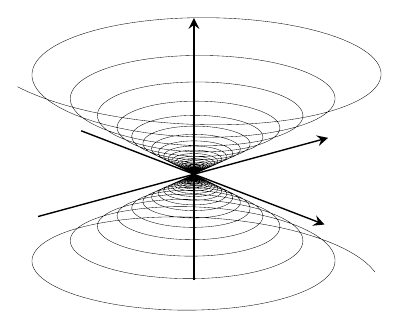}
			\caption{\label{fig:whirpool}The image of $\mathbb{R}$ under $f=[f_1,f_2,f_3]$.}
		\end{figure}
	\end{minipage}\\[0.5cm]	
	We will show, that $f_1,f_2,f_3$ are functionally dependent in the sense of Definition \ref{def:Laszlo} on any open set $\Omega$ which contains $x_0=0$ and are functionally independent at $x_0$ in the sense of Definition \ref{def:manifolds}. 	
	
	Derivatives of  functions $f_1,f_2,f_3$ are as follows
	\begin{align*}
	&f_1^\prime(x)=\left\{\begin{array}{ll}
	3x^2 \sin(\frac{1}{x}) - x\cos (\frac{1}{x}) & \text{if}\ x\neq 0,\\
	0 & \text{otherwise}
	\end{array}\right.\\
	&f_2^\prime(x)=\left\{\begin{array}{ll}
	3x^2 \cos(\frac{1}{x}) + x\sin (\frac{1}{x}) & \text{if}\ x\neq 0,\\
	0 & \text{otherwise}
	\end{array}\right.\\
	&f_3^\prime(x)=3x^2.
	\end{align*}
	For any open set $U_0$ containing $x_0=0$ we have $\rank \{ f_1^\prime(x),f_2^\prime(x),f_3^\prime(x)\}<3$, $x\in U_0$. Hence, functions $f_1,f_2,f_3$ are functionally dependent on $U_0$ in the sense of Definition \ref{def:Laszlo}. 
	
	Let $U(x_0)$ be any neighbourhood of $x_0$ and let $F:\ \mathbb{R}^3\rightarrow \mathbb{R}$ be any $C^1$ function such that $F(f(U(x_0)))=0$, where $f(x)=[f_1(x),f_2(x),f_3(x)]$. Let $y_0=f(x_0)=(0,0,0)$ and $t_n=\frac{1}{\frac{\pi}{2}+2n \pi }$. Then
	$$F(t_n^3[1,0,1])=F(t_n^3,0,t_n^3)=F(f_1(t_n),f_2(t_n),f_3(t_n))=0$$ for sufficiently large $n$ and
	\begin{equation*}
	F^\prime(y_0,[1,0,1])=\lim_{t\rightarrow 0} \frac{F(t[1,0,1])}{t}=\lim_{n\rightarrow +\infty} \frac{F(t_n^3[1,0,1])}{t_n^3}= \lim_{n\rightarrow +\infty} \frac{0}{t_n^3}=0.
	\end{equation*}
	Analogously one can show that
	\begin{equation*}
	F^\prime(y_0,[0,1,1])=0,\quad 
	F^\prime(y_0,[-1,0,1])=0.
	\end{equation*}
	Hence $D F(y_0)=0$, i.e functions $f_1,f_2,f_3$ cannot be functionally dependent at $x_0$ in the sense of Definition \ref{def:manifolds}.
\end{example}

\begin{remark}\label{remark:whirpool}
	Let $f=[f_1,\dots, f_\kappa]:\ \mathbb{R}^k\rightarrow\mathbb{R}^\kappa$ be of class $C^1$ in a neighbourhood of $x_0\in \mathbb{R}^k$. 
	If for any neighbourhood  $U(x_0)$ there exist $v_1,\dots,v_\kappa\in \mathbb{R}^\kappa$ linearly independent and a sequence $t_m>0$, $t_m\rightarrow0$ such that $ f( U(x_0))\cap (f(x_0)+t_m v_i)\neq \emptyset$ for all $m\in \mathbb{N}$, then functions $f_1,\dots, f_\kappa$ are functionally independent at $x_0$ in the sense of Definition \ref{def:manifolds}.
	
\end{remark}
It is  clear, that Proposition \ref{proposition:functional_independence2} implies Remark \ref{remark:whirpool}, but the  converse does not hold (as in Example \ref{example:whirpool}). 


In conclusion, the established relationships between above concepts of functional dependence at $x_0$ are illustrated in the following graph.

\begin{figure}[H]
	\centering
	\begin{tikzpicture}[node distance=2cm, auto]\
	\node[draw] at (0,0) (DefBoothby) [text width=2.5cm] {Definition \ref{def:Boothby}
	};
	\node[draw] at (0,2) (DefManifolds) [text width=2.5cm] {Definition \ref{def:manifolds}};
	\node[draw] at (7,2) (DefRaghavan) [text width=2.5cm] {Definition \ref{def:Raghavan} };
	\node[draw] at (7,0) (DefLaszlo) [text width=3.5cm] {Definition \ref{def:Laszlo} on a neighbourhood of $x_0$};
	\draw[->] (DefManifolds) to node {} (DefBoothby);
	\draw[->] (DefRaghavan) to [in=40,out=-130] node [xshift=0.3cm]{Proposition \ref{propostion:Raghavan_imples_Boothby}} (DefBoothby);
	\draw[<->,dashed] (DefLaszlo) to  node {\cite[Theorem 4.1.3]{Ordinary_and_partial_diff-eq_for_beginner_Laszlo} } (DefBoothby);		
	\draw[->] (DefLaszlo) to [in=-45,out=130] node [xshift=-0.5cm,yshift=0.85cm]{
		\begin{tabular}{c}
		Example \ref{example:whirpool}\\
		$\not$
		\end{tabular} }  (DefManifolds);
	\draw[->] (DefBoothby) to [out=130, in=230] node [yshift=0.1cm]{
		Example \ref{example:functional_dependence} $\not$}  (DefManifolds);
	\draw[->] (DefRaghavan) to [out=-90, in=90] node [yshift=0.15cm]{
		Proposition \ref{proposition:Raghavan_to_Laszlo}}  (DefLaszlo);
	\draw[->] (DefRaghavan) to [out=160, in=25] node [yshift=0.75cm]{
		\begin{tabular}{c}
		Example \ref{example:not_manifolds}\\
		$\not$
		\end{tabular}}  (DefManifolds);
	\end{tikzpicture}
	\caption{ Relationships between functional dependence concepts. Theorem 4.1.3 of \cite{Ordinary_and_partial_diff-eq_for_beginner_Laszlo} is given in the case $f_1,\dots,f_\kappa:\ \mathbb{R}^\kappa\rightarrow\mathbb{R}$. 
	} 
\end{figure}
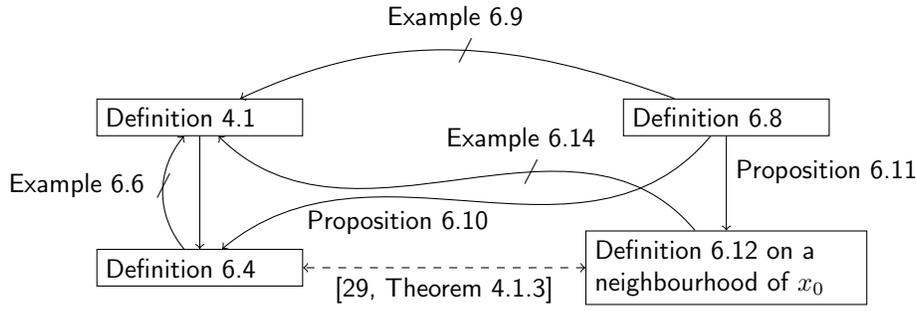

\section{Relaxed Constant Rank Constraint Qualification and Lagrange multipliers}\label{section:relaxed_constant_rank_and_lagrange_multipliers}

In this section we prove the non-emptiness of Lagrange multipliers set at a local minimum under RCRCQ and as an application we show that the linearized problem \eqref{problem:P} at a local minimum is solvable.

Let us consider the problem \eqref{problem:P},
\begin{equation}\tag{P}
\begin{array}{ll}
\text{minimize } & h_0(x) \\
\text{s.t. } & x\in {\mathcal F},
\end{array}
\end{equation}
where ${\mathcal F}$ is of the form \eqref{set:F}, i.e. 
\begin{equation}
\begin{array}{ll}
{\mathcal F}:=\{ x\in \hilbertH \mid h_i(x)=0, i\in I_0,\ h_i(x)\leq 0,\ i\in I \},
\end{array}
\end{equation}
where $\hilbertH$ is a Banach space,  $h_0, h_i:\ \hilbertH \rightarrow \mathbb{R}$, $i\in I_0\cup I$ are of class $C^1$ in a neighbourhood of a given local minimum $x_0\in {\mathcal F}$ of $h_0$ and sets $I_0, I$ are finite ($I_0\cup I=\{1,\dots,n\}$).

Let $\Lambda(x)$ be the set of Lagrange multipliers at  $x\in {\mathcal F}$, i.e.
\begin{equation*}
\Lambda(x):=\{\lambda\in \mathbb{R}^n \mid D_x L(\lambda,x)=0,\ \lambda_i\geq 0 \text{ and } \lambda_i h_i(x)=0,\ i\in I  \},
\end{equation*}
where
\begin{equation*}
L(\lambda,x):=h_0(x)+\sum_{i=1}^n \lambda_i  h_i(x), \quad \lambda=(\lambda_1,\dots,\lambda_n)
\end{equation*}
is the Lagrangean to problem \eqref{problem:P}. The following proposition relates RCRCQ to the existence of Lagrange multipliers. More on this topic see \cite{doi:10.1080/02331934.2019.1696339}.

\begin{proposition}\label{proposition:nonmpty_lagrange}
	Let $x_0\in {\mathcal F}$ be a local minimum of problem \eqref{problem:P} and let  RCRCQ hold for ${\mathcal F}$ at $x_0$. Then $\Lambda(x_0)\neq \emptyset$.
\end{proposition}

\begin{proof}
	By Theorem \ref{theorem:tangent_cone}, for any $d\in \Gamma_{\mathcal F} (x_0)$ there exists a vector function $o(t)$ such that $\lim_{t\rightarrow 0} \|o(t)\|t^{-1}=0$  and $x_0+td+o(t)\in {\mathcal F}$. Since $x_0$ is a local minimum of $h_0$ on ${\mathcal F}$ we have  $h_0(x_0+td+o(t))-h_0(x_0)\geq 0$ for all $t$ sufficiently small. By Taylor expansion, we have
	\begin{align*}
	0&\leq h_0(x_0+td+o(t))-h_0(x_0)\\
	&=h_0(x_0)+\langle D h_0(x_0+\theta(td+o(t))) \, , \, td+o(t)\rangle-h_0(x_0) \\
	&=t\langle D h_0(x_0+\theta(td+o(t))) \, , \, d\rangle+\langle D h_0(x_0+\theta(td+o(t))) \, , \, o(t)\rangle,
	\end{align*}
	where $\theta\in [0,1]$ and $\theta$ depends on $t,d$. Hence,
	\begin{equation}\label{ineq:local_min}
	\langle D h_0(x_0+\theta(td+o(t))) \, , \, d\rangle\geq -\langle D h_0(x_0+\theta(td+o(t))) \, , \, o(t)t^{-1}\rangle.
	\end{equation}
	By passing to the limit with $t\rightarrow 0$ in \eqref{ineq:local_min} we obtain $\langle D h_0(x_0) \, , \, d \rangle \geq 0$.
	Hence 
	\begin{equation}
	\label{nabla}
	-D h_0(x_0) \in (\Gamma_{\mathcal F}(x_0))^\circ,
	\end{equation} where $ (\Gamma_{\mathcal F}(x_0))^\circ$ is dual cone defined as $$ (\Gamma_{\mathcal F}(x_0))^\circ :=\{d^*\in \hilbertH^* \mid \langle d^*, d\rangle \le 0, \forall d \in  \Gamma_{\mathcal F}(x_0)\}.$$
	
	Since
	\begin{equation*}
	\Gamma_{\mathcal F}(x_0)=\left\{ d\in \hilbertH\ \bigg|\  	\begin{array}{ll}
	\langle D h_i(x_0) \, , \, d \rangle \leq 0,  & i\in I(x_0),\\
	\langle D h_i(x_0) \, , \, d \rangle \leq 0,   & i\in I_0,\\
	\langle -D h_i(x_0) \, , \, d \rangle \leq 0,   & i\in I_0
	\end{array}\right\},
	\end{equation*}
	by 	\cite[Theorem 6.40]{deutsch2001best},
	the dual cone to $\Gamma_{\mathcal F}(x_0)$ is given as follows
	\begin{align}
	\label{nabla2}
	\begin{aligned}
	(\Gamma_{\mathcal F}(x_0))^\circ = \{ d^*\in \hilbertH^* \mid d^*=&\sum_{i\in I_0\cup I(x_0)} \lambda_i  D h_i(x_0),\\
	& \lambda_i\geq 0,\ i \in I(x_0),\ \lambda_i \in \mathbb{R},\ i\in I_0   \}.
	\end{aligned}
	\end{align}
	By \eqref{nabla} and \eqref{nabla2}, there exist $\lambda_i\geq 0,\ i \in I(x_0),\ \lambda_i \in \mathbb{R},\ i\in I_0$ such that 
	$$-D h_0(x_0) =\sum_{i\in I_0\cup I(x_0)} \lambda_i  D h_i(x_0).
	$$
	By putting $\lambda_i=0$ for $i\in \{1,2,\dots,n\}\setminus (I_0\cup I(x_0))$, we have $\Lambda(x_0)\neq \emptyset.$		
\end{proof}

As an application of Proposition \ref{proposition:nonmpty_lagrange} we show that the linearized problem,  
at a local minimum to \eqref{problem:P}, $x_0 \in {\mathcal F}$, 
\begin{equation}\label{problem:primal}
\begin{array}{lll}
\text{minimize} & \langle D h_0(x_0) \, , \, d \rangle\\
\text{s.t.} & \langle D h_i(x_0) \, , \, d \rangle \leq 0,  & i\in I(x_0),\\
& \langle D h_i(x_0) \, , \, d \rangle =0 ,  & i\in I_0
\end{array}
\end{equation}
is solvable with a solution $d=0$.
Problem \eqref{problem:primal} can be equivalently rewritten as
\begin{equation}\label{problem:primal2}
\begin{array}{lll}
\text{minimize} & \langle D h_0(x_0) \, , \, d \rangle\\
\text{s.t.} & \langle D h(x_0) \, , \, d \rangle \in K,\\
\end{array}
\end{equation}	
where $h=[h_i]_{i\in I(x_0)\cup I_0}$, $K=\{ k=(k_i) \in \mathbb{R}^{|I_0 \cup I(x_0)|} \mid k_i\leq 0, i\in I(x_0),\ k_i=0,\ i \in I_0  \}$.
Lagrangian to \eqref{problem:primal2} is defined as follows
\begin{equation}
L(d,\lambda)=\langle D h_0(x_0) \, , \, d \rangle+\langle \lambda \, , \, D h(x_0)  d  \rangle,
\end{equation}
where $\lambda \in \mathbb{R}^{|I_0\cup I(x_0)|}$.
Let $K^*:=\{ k^*=(k_i^*) \in  \mathbb{R}^{|I_0\cup I(x_0)|} \mid \langle k \, , \, k^* \rangle \leq 0 \text{ for all } k\in K  \}=\{ k^* \in  \mathbb{R}^{|I_0\cup I(x_0)|}  \mid k_i^*\geq 0,\ i\in I(x_0),\ k_i^*\in \mathbb{R},\ i \in I_0    \} .$
The dual to \eqref{problem:primal2} takes the form
\begin{equation}\label{problem:dual}
\begin{array}{ll}
\text{maximize}& \inf_{d\in \hilbertH} L(d,\lambda )\\
\text{s.t}& \lambda \in K^*.
\end{array}		
\end{equation}
Let us consider the objective of the dual. We have
\begin{align*}
\inf_{d\in \hilbertH} L(d,\lambda )&=\inf_{d\in \hilbertH} \left\{ \langle D h_0(x_0) \, , \, d \rangle+\langle \lambda \, , \,  D h(x_0)  d  \rangle\right\}\\
&=\inf_{d\in \hilbertH} \left\{ \langle D h_0(x_0) \, , \, d \rangle+\sum_{i\in I_0\cup I(x_0)} \lambda_i  \langle D h_i(x_0) \, , \,  d  \rangle\right\}\\
&=\inf_{d\in \hilbertH} \left\{\langle D h_0(x_0)+\sum_{i\in I_0\cup I(x_0)} \lambda_i D h_i(x_0) \, , \, d \rangle\right\}.
\end{align*}
Hence,
\begin{align*}
\inf_{d\in \hilbertH} \langle D h_0(x_0)&+\sum_{i\in I_0\cup I(x_0)} \lambda_i D h_i(x_0) \, , \, d \rangle\\
&=\left\{ \begin{array}{ll}
-\infty & \text{if}\  D h_0(x_0)+\sum_{i\in I_0\cup I(x_0)} \lambda_i D h_i(x_0)\neq 0, \\
0 & \text{if}\  D h_0(x_0)+\sum_{i\in I_0\cup I(x_0)} \lambda_i D h_i(x_0)=0.
\end{array}\right.
\end{align*}
Thus \eqref{problem:dual} is equivalent to the following
\begin{equation}\label{problem:dual2}
\begin{array}{ll}
\text{maximize}& 0\\
\text{s.t}&  D h_0(x_0)+\sum_{i\in I_0\cup I(x_0)} \lambda_i D h_i(x_0)=0,\\
&\lambda \in K^*.
\end{array}		
\end{equation}
In conclusion, we obtain the following corollary.
\begin{corollary}
	Under assumption of RCRCQ at $x_0\in {\mathcal F}$, where $x_0$ is a local minimum of \eqref{problem:P}, the element $d=0$ is a solution of \eqref{problem:primal}, since, by Proposition \ref{proposition:nonmpty_lagrange}, feasible set of \eqref{problem:dual2} is nonempty.
\end{corollary}

\bibliographystyle{plain}
\bibliography{mybibfile}  

\end{document}